\documentclass[11pt]{article}

\usepackage[T1]{fontenc}
\usepackage{lmodern}

\usepackage{amsmath, amssymb, amsthm} 
\usepackage[margin=0.96in]{geometry}
\usepackage{color, graphicx}

\usepackage{etoolbox} 
\makeatletter
\patchcmd{\@maketitle}{\LARGE \@title}{\LARGE\bfseries\@title}{}{}

\renewcommand{\@seccntformat}[1]{\csname the#1\endcsname.\quad}
\makeatother

\definecolor{darkblue}{rgb}{0,0,.5}
\usepackage{hyperref}
\hypersetup{
	colorlinks=true,        % false: boxed links; true: colored links
	linkcolor=darkblue,     % color of internal links
	citecolor=darkblue,     % color of links to bibliography
	urlcolor=darkblue       % color of external links
}

\makeatletter
\def\th@plain{%
	\thm@notefont{}% same as heading font
	\itshape % body font
}
\def\th@definition{%
	\thm@notefont{}% same as heading font
	\normalfont % body font
}

\renewenvironment{proof}[1][\proofname]{\par
	\normalfont
	\topsep0\p@\@plus3\p@ \trivlist
	\item[\hskip\labelsep\itshape
	#1\@addpunct{.}]\ignorespaces
}{%
	\qed\endtrivlist
}
\makeatother

\newtheorem{theorem}{Theorem}[section]
\newtheorem{lemma}[theorem]{Lemma}
\newtheorem{corollary}[theorem]{Corollary}
\newtheorem{proposition}[theorem]{Proposition}
\newtheorem{fact}[theorem]{Fact}

\theoremstyle{definition}
\newtheorem{definition}[theorem]{Definition}
\theoremstyle{definition}

\theoremstyle{definition}
\newtheorem{remark}[theorem]{Remark}

\usepackage[shortlabels]{enumitem}

\allowdisplaybreaks

\newcommand{\menge}[2]{\{{#1}~\big |~{#2}\}} 
 
\newcommand{\Menge}[2]{\left\{{#1}~\Big|~{#2}\right\}} 

\newcommand{\scal}[2]{\left\langle {#1},{#2} \right\rangle}

\newcommand{\NN}{\ensuremath{\mathbb N}}
\newcommand{\nnn}{\ensuremath{{n\in{\mathbb N}}}}

\newcommand{\RR}{\ensuremath{\mathbb R}}

\newcommand{\RP}{\ensuremath{\mathbb{R}_+}}
\newcommand{\RPP}{\ensuremath{\mathbb{R}_{++}}}

\newcommand{\argmin}{\ensuremath{\operatorname*{argmin}}}

\newcommand{\zer}{\ensuremath{\operatorname{zer}}}
\newcommand{\dom}{\ensuremath{\operatorname{dom}}}

\newcommand{\gra}{\ensuremath{\operatorname{gra}}}
\newcommand{\Fix}{\ensuremath{\operatorname{Fix}}}
\newcommand{\Id}{\ensuremath{\operatorname{Id}}}
\newcommand{\prox}{\ensuremath{\operatorname{Prox}}}

\begin{document}

\title{Conical averagedness and convergence analysis\\ of fixed point algorithms}

\author{
Sedi Bartz\thanks{Department of Mathematical Sciences, Kennedy College of Sciences, University of Massachusetts Lowell, Lowell, MA 01854, USA.
E-mail: \texttt{sedi\char`_bartz@uml.edu}.},
~
Minh N.\ Dao\thanks{School of Engineering, Information Technology and Physical Sciences, Federation University Australia, Ballarat, VIC 3353, Australia.  
E-mail: \texttt{m.dao@federation.edu.au}},
~and~
Hung M.\ Phan\thanks{Department of Mathematical Sciences, Kennedy College of Sciences, University of Massachusetts Lowell, Lowell, MA 01854, USA.
E-mail: \texttt{hung\char`_phan@uml.edu}.}}

\date{November 27, 2020}

\maketitle

\begin{abstract}
We study a conical extension of averaged nonexpansive operators and the role it plays in convergence analysis of fixed point algorithms. Various properties of conically averaged operators are systematically investigated, in particular, the stability under relaxations, convex combinations and compositions. We derive conical averagedness properties of resolvents of generalized monotone operators. These properties are then utilized in order to analyze the convergence of the proximal point algorithm, the forward-backward algorithm, and the adaptive Douglas--Rachford algorithm. Our study unifies, improves and casts new light on recent studies of these topics.
\end{abstract}

{\small

}

\section{Introduction}

Averaged nonexpansive operators, originally introduced in \cite{BBR78}, are well known to be useful in convergence analysis of various fixed point algorithms, see \cite{BC17,BNP15,Com04,CY15,KRZ17,OY02} and the references therein.
In particular, iterative sequences generated by several fixed point algorithms can be expressed in terms of Krasnosel'ski\u{\i}--Mann iterations \cite{Krasnosel55, Mann53}, the convergence of which relies upon an averagedness property. Although frequently understood in the single-valued setting, some averaged nonexpansive properties can also be explored in the set-valued framework. For instance, the notion of \emph{union averaged nonexpansive operators} has been recently studied in \cite{DT19} with applications to the local convergence of proximal algorithms; see also \cite{Tam18}.

Each averaged operator is an under-relaxation of some nonexpansive operator. We demonstrate that over-relaxations of nonexpansive operators also arise naturally in several situations. In our study, we consider \emph{conically averaged operators} and provide a framework that unifies both types of relaxations for nonexpansive operators. We then show that this class of operators plays a significant role in convergence analysis of several fixed point algorithms, in particular, the \emph{proximal point algorithm}, the \emph{forward--backward algorithm}, and the \emph{adaptive Douglas--Rachford (DR) algorithm}.

The proximal point algorithm was introduced by Martinet~\cite{Mar70} and further developed by Rockafellar \cite{Roc76} for finding a zero of a maximally monotone operator. This was followed by other early studies such as~\cite{BR78}. Since then, the proximal point algorithm turned into an indispensable tool of optimization, both in theory and in applications. In fact, several iterative optimization algorithms can be reformulated as special cases of the proximal point algorithm, see, for example,~\cite{EB92} and references therein. On the other hand, the forward-backward algorithm was first proposed by Lions and Mercier \cite{LM79} and Passty~\cite{Pas79} for finding a zero of the sum of two maximally monotone operators. This splitting idea can be traced back to the projected gradient method~\cite{Gol64}. Another well-known splitting algorithm is the DR algorithm, initially studied by Douglas and Rachford~\cite{DR56} in the setting of linear operators and was later generalized for maximally monotone operators by Lions and Mercier, also in~\cite{LM79}. It is worth mentioning that both the forward-backward and the DR algorithms reduce to the proximal point algorithm in the case where one operator is zero. Recently, the so-called adaptive DR algorithm has been proposed in \cite{DP18a} to deal with scenarios that lack classical monotonicity.

The main objectives of our study are as follows.
\begin{enumerate}
\item 
We systematically study the conical averagedness property and prove that it is stable under relaxations, convex combinations and compositions. In order to establish these properties, we extend known results from \cite[Section~4.5]{BC17} and \cite[Section~2]{CY15} for nonexpansive and averaged operators. In particular, under suitable conditions, in Theorem~\ref{t:finite_comp} we show that compositions of scaled conically averaged operators are also conically averaged. Furthermore, we show that relaxed resolvents of generalized monotone operators either belong to or directly linked to the class of conically averaged operators.
\item 
Under generalized monotonicity assumptions, we provide convergence analysis by means of conical averagedness, in particular, in Theorem~\ref{t:pp} for the relaxed proximal point algorithm, in Theorem~\ref{t:FB} for the relaxed forward-backward algorithm, in Theorems~\ref{t:2cocoer} and~\ref{t:2mono} for the adaptive DR algorithm. We emphasize that these splitting algorithms incorporate several operators, some of which may not be averaged due to the lack of classical monotonicity. Nevertheless, we prove that their compositions are, in fact, conically averaged. Consequently, we derive convergence and rate of asymptotic regularity.
This approach sheds light on the structures of these algorithms and yields simple and transparent convergence proofs. An application to the strongly-weakly convex optimization problem is included in Theorem~\ref{t:f-cvg}. Our analysis improves several contemporary results on this topic, e.g., \cite[Proposition~26.1]{BC17},
\cite[Theorems~4.5 and 5.4]{DP18a}, and \cite[Theorems~4.4 and 4.6]{GHY17}.
\end{enumerate}
 
The remainder of this paper is organized as follows. In Section~\ref{s:ext_avg} we present conically averaged operators as well as several interesting properties which are useful for our analysis. In Section~\ref{s:gen_mono_coco} we study relaxed resolvents of generalized monotone operators in relation with their conical averagedness properties. In Sections~\ref{s:rfb} and~\ref{s:adr} we provide our main results regarding conical averagedness of operators associated with the previously mentioned algorithms, which, in turn, lead to their convergence and rate of asymptotic regularity.

We will employ standard notations that generally follow \cite{BC17}. Throughout, $X$ is a real Hilbert space with inner product $\scal{\cdot}{\cdot}$ and induced norm $\|\cdot\|$. 
The set of nonnegative integers is denoted by $\NN$, the set of real numbers by $\RR$, the set of nonnegative real numbers by $\RP := \menge{x \in \RR}{x \geq 0}$, and the set of the positive real numbers by $\RPP := \menge{x \in \RR}{x >0}$.
We use the notation $A\colon X\rightrightarrows X$ to indicate that $A$ is a set-valued operator on $X$ and the notation $A\colon X\to X$ to indicate that $A$ is a single-valued operator on $X$. We denote the \emph{domain} of the mapping $A:X\rightrightarrows X$ by $\dom A :=\menge{x\in X}{Ax\neq \varnothing}$, its set of \emph{zeros} by $\zer A :=\menge{x\in X}{0\in Ax}$, and its set of \emph{fixed points} by $\Fix A :=\menge{x\in X}{x\in Ax}$. $\Id$ denotes the \emph{identity mapping}.

We will often use the following identity. For every $\sigma,\tau\in\RR$ and $s,t\in X$,
\begin{equation}
\label{e:identity}
\|\sigma s +\tau t\|^2=\sigma(\sigma+\tau)\|s\|^2
+\tau(\sigma+\tau)\|t\|^2-\sigma\tau\|s-t\|^2;
\end{equation}
moreover, if $\sigma+\tau\neq 0$, then
\begin{equation}
\label{e:identity2}
\sigma\|s\|^2 +\tau\|t\|^2=
\frac{1}{\sigma+\tau}\|\sigma s+\tau t\|^2
+\frac{\sigma\tau}{\sigma+\tau}\|s-t\|^2.
\end{equation}

\section{Conically averaged operators}
\label{s:ext_avg}

We recall that the mapping $T\colon X\to X$ is \emph{nonexpansive} if it is Lipschitz continuous with constant $1$ on its domain, i.e.,
\begin{equation}
\forall x, y\in \dom T,\quad \|Tx-Ty\|\leq \|x-y\|.
\end{equation}
$T$ is said to be \emph{$\theta$-averaged} if $\theta\in \left]0,1\right[$ and $T =(1-\theta)\Id +\theta N$ for some nonexpansive operator $N\colon X\to X$, see, e.g., \cite[Definition~4.33]{BC17}. We now extend this concept to allow for $\theta\in \RPP$.

\begin{definition}[conically averaged operator]
We say that an operator $T\colon X\rightarrow X$ is \emph{conically averaged} with constant $\theta\in \RPP$, or \emph{conically $\theta$-averaged}, if there exists a nonexpansive operator $N\colon X\to X$ such that 
\begin{equation}
T =(1-\theta)\Id +\theta N.
\end{equation}
\end{definition}
Let $T$ be conically $\theta$-averaged. Then $T$ is nonexpansive when $\theta =1$, and $T$ is $\theta$-averaged when $\theta\in \left]0,1\right[$. As one would expect, conically averaged operators also possess properties similar to averaged operators. Indeed, we now present several properties which generalize similar properties from \cite{BC17,CY15}, see also \cite{BMW19} for a related development where conically averaged operators were referred to as conically nonexpansive operators.

\begin{proposition}
\label{p:ext_avg}
Let $T\colon X\to X$, $\theta\in \RPP$, and $\lambda\in \RPP$. Then the following assertions are equivalent:
\begin{enumerate}
\item\label{p:ext_avg_ori}
$T$ is conically $\theta$-averaged.
\item\label{p:ext_avg_rlx} 
$(1-\lambda)\Id +\lambda T$ is conically $\lambda\theta$-averaged.  
\item\label{p:ext_avg_ine} 
For all $x,y\in \dom T$, 
\begin{equation}
\|Tx-Ty\|^2\leq \|x-y\|^2 -\frac{1-\theta}{\theta}\|(\Id-T)x -(\Id-T)y\|^2.
\end{equation} 
\item\label{p:ext_avg_ine'} 
For all $x,y\in \dom T$, 
\begin{equation}
\|Tx-Ty\|^2 +(1-2\theta)\|x-y\|^2\leq 2(1-\theta)\scal{x-y}{Tx-Ty}.
\end{equation} 
\end{enumerate}
\end{proposition}
\begin{proof}
Set $N :=(1-1/\theta)\Id +(1/\theta)T$. Then $T =(1-\theta)\Id +\theta N$ and $(1-\lambda)\Id +\lambda T =(1-\lambda\theta)\Id +\lambda\theta N$. By definition,
\begin{subequations}
\begin{align}
T \text{~is conically $\theta$-averaged}
&\iff N \text{~is nonexpansive}\\
&\iff (1-\lambda)\Id +\lambda T \text{~is conically $\lambda\theta$-averaged},
\end{align}
\end{subequations}
which implies the equivalence between \ref{p:ext_avg_ori} and \ref{p:ext_avg_rlx}.

Next, we note that $\Id-N =(\Id-T)/\theta$ and $\dom T =\dom N =:D$. Now, by employing~\eqref{e:identity}, for all $x,y\in D$
\begin{subequations}
\begin{align}
\|Tx-Ty\|^2 &=\|(1-\theta)(x-y) +\theta(Nx-Ny)\|^2\\
&=(1-\theta)\|x-y\|^2 +\theta\|Nx-Ny\|^2
 -\theta(1-\theta)\|(\Id-N)x -(\Id-N)y\|^2\\
&=\|x-y\|^2 -\frac{1-\theta}{\theta}\|(\Id-T)x -(\Id-T)y\|^2+\theta(\|Nx-Ny\|^2 -\|x-y\|^2).
\end{align}
\end{subequations} 
Consequently, we see that 
\begin{subequations}
\begin{align}
& T \text{~is conically $\theta$-averaged}\\
\iff\ & N \text{~is nonexpansive}\\
\iff\ & \forall x,y\in D,\quad \|Nx-Ny\|\leq \|x-y\|\\
\iff\ & \forall x,y\in D,\quad \|Tx-Ty\|^2\leq \|x-y\|^2 -\frac{1-\theta}{\theta}\|(\Id-T)x -(\Id-T)y\|^2\\
\iff\ & \forall x,y\in D,\quad \|Tx-Ty\|^2 +(1-2\theta)\|x-y\|^2\leq 2(1-\theta)\scal{x-y}{Tx-Ty},
\end{align}
\end{subequations}
and we get the equivalence between \ref{p:ext_avg_ori}, \ref{p:ext_avg_ine}, and \ref{p:ext_avg_ine'}. The proof is complete. 
\end{proof}

In view of Proposition~\ref{p:ext_avg}, we see that $T$ is $1/2$-averaged if and only if, for all $x,y\in \dom T$ 
\begin{equation}
\|Tx-Ty\|^2\leq \|x-y\|^2 -\|(\Id-T)x -(\Id-T)y\|^2,
\end{equation} 
in which case $T$ is said to be \emph{firmly nonexpansive} (see also \cite[Proposition~11.2]{GR84}).

\begin{proposition}
\label{p:ext_averaged2}
Let $T\colon X\to X$ and $\lambda\in \RPP$. Then $T$ is firmly nonexpansive if and only if $\Id -\lambda T$ is conically $\lambda/2$-averaged.
\end{proposition}
\begin{proof}
Set $N :=2T -\Id$. Then $T =(1/2)\Id +(1/2)N$ is firmly nonexpansive $\iff$ $N$ is nonexpansive $\iff$ $\Id -\lambda T =(1-\lambda/2)\Id +(\lambda/2)(-N)$ is conically $\lambda/2$-averaged.
\end{proof}

The following results reiterate and extend similar facts regarding nonexpansive and averaged operators from \cite[Section~4.5]{BC17} and \cite[Section~2]{CY15}.

\begin{proposition}[convex combination of conically averaged operators]
Let $I$ be a finite index set. For each $i\in I$ let $T_i\colon X\to X$ be conically $\theta_i$-averaged. Let $\{w_i\}_{i\in I}\subseteq \RPP,\ \sum_{i\in I} \omega_i =1$. Then $\sum_{i\in I}\omega_i T_i$ is conically $\theta$-averaged where $\theta:=\sum_{i\in I} \omega_i\theta_i$.
\end{proposition}
\begin{proof}
For each $i\in I$, there exists a nonexpansive operator $N_i$ such that $T_i=(1-\theta_i)\Id+\theta_i N_i$. It follows that
\begin{equation}
\sum_{i\in I} \omega_i T_i=
\sum_{i\in I} \omega_i(1-\theta_i)\Id
+\sum_{i\in I} \omega_i\theta_i N_i
=(1-\theta)\Id +\theta\sum_{i\in I} \frac{\omega_i\theta_i}{\theta} N_i.
\end{equation}
Since, by \cite[Proposition~4.9(i)]{BC17}, $\sum_{i\in I} \frac{\omega_i\theta_i}{\theta} N_i$ is nonexpansive, the proof is complete.
\end{proof}

\begin{proposition}[composition of two conically averaged operators]
\label{p:2averaged}
Let $T_1\colon X\to X$ and $T_2\colon X\to X$ be conically $\theta_1$-averaged and conically $\theta_2$-averaged, respectively. Suppose that either $\theta_1 =\theta_2 =1$ or $\theta_1\theta_2 <1$. Let $\omega\in\RR\smallsetminus\{0\}$ and set
\begin{equation}\label{e:2averaged}
T :=\Big(\frac{1}{\omega}T_2\Big)(\omega T_1)
\quad\text{and}\quad
\theta :=\begin{cases}
1 &\text{if~} \theta_1 =\theta_2 =1,\\
\frac{\theta_1+\theta_2-2\theta_1\theta_2}{1-\theta_1\theta_2} &\text{if~} \theta_1\theta_2 <1.
\end{cases}
\end{equation}
Then $T$ is conically $\theta$-averaged. Moreover, 
\begin{enumerate}
\item\label{p:2averaged_theta=1}
$\theta =1$ if and only if $\theta_1 =1$ or $\theta_2 =1$.
\item\label{p:2averaged_theta<1}  
$\theta<1$ if and only if $\theta_1<1$ and $\theta_2<1$.
\end{enumerate}
\end{proposition}
\begin{proof}
Let $x,y\in \dom T_2(\omega T_1)$. By applying Proposition~\ref{p:ext_avg} to $T_2$ and then to $T_1$, we obtain
\begin{subequations}
\label{e:T2T1}
\begin{align}
&
\Big\|(\tfrac{1}{\omega}T_2)(\omega T_1)x -(\tfrac{1}{\omega}T_2)(\omega T_1)y\Big\|^2
=\frac{1}{\omega^2}\|T_2(\omega T_1 x)-T_2(\omega T_1y)\|^2\\
&\leq \frac{1}{\omega^2}\left(\|\omega T_1x-\omega T_1y\|^2 -\frac{1-\theta_2}{\theta_2}\|(\Id-T_2)(\omega T_1x) -(\Id-T_2)(\omega T_1y)\|^2\right)\\
&= \|T_1x-T_1y\|^2 -\frac{1-\theta_2}{\theta_2}\Big\|\Big(T_1-(\tfrac{1}{\omega}T_2)(\omega T_1)\Big)x -\Big(T_1-(\tfrac{1}{\omega}T_2)(\omega T_1)\Big)y\Big\|^2\\
&\leq \|x-y\|^2 -\frac{1-\theta_1}{\theta_1}\|(\Id-T_1)x -(\Id- T_1)y\|^2 \notag\\ 
&\qquad\qquad\quad-\frac{1-\theta_2}{\theta_2}
\Big\|\Big(T_1-(\tfrac{1}{\omega}T_2)(\omega T_1)\Big)x -\Big(T_1-(\tfrac{1}{\omega}T_2)(\omega T_1)\Big)y\Big\|.
\end{align}
\end{subequations}

\emph{Case 1:} If $\theta_1 =\theta_2 =1$, then \eqref{e:T2T1} implies that $(\frac{1}{\omega}T_2)(\omega T_1)$ is nonexpansive, i.e., conically $1$-averaged. 

\emph{Case 2:} If $\theta_1\theta_2<1$, then 
\begin{equation}
\frac{1}{\theta_1} +\frac{1}{\theta_2}\geq \frac{2}{\sqrt{\theta_1\theta_2}}>2,
\end{equation}
which implies
\begin{equation}\label{e:181030a}
\frac{1-\theta_1}{\theta_1} +\frac{1-\theta_2}{\theta_2}>0.
\end{equation}
By setting $s :=(\Id-T_1)x -(\Id-T_1)y$ and $t :=\Big(T_1-(\tfrac{1}{\omega}T_2)(\omega T_1)\Big)x -\Big(T_1-(\tfrac{1}{\omega}T_2)(\omega T_1)\Big)y$ and by employing \eqref{e:identity2} and \eqref{e:181030a} we see that
\begin{subequations}
\begin{align}
\frac{1-\theta_1}{\theta_1}\|s\|^2 +\frac{1-\theta_2}{\theta_2}\|t\|^2 &=\frac{\frac{1-\theta_1}{\theta_1}\frac{1-\theta_2}{\theta_2}}{\frac{1-\theta_1}{\theta_1}+\frac{1-\theta_2}{\theta_2}}\|s+t\|^2 +\frac{1}{\frac{1-\theta_1}{\theta_1}+\frac{1-\theta_2}{\theta_2}}\left\|\frac{1-\theta_1}{\theta_1}s-\frac{1-\theta_2}{\theta_2}t\right\|^2\\
&\geq \frac{\frac{1-\theta_1}{\theta_1}\frac{1-\theta_2}{\theta_2}}{\frac{1-\theta_1}{\theta_1}+\frac{1-\theta_2}{\theta_2}} \Big\|\Big(\Id-(\tfrac{1}{\omega}T_2)(\omega T_1)\Big)x -\Big(\Id-(\tfrac{1}{\omega}T_2)(\omega T_1)\Big)y\Big\|^2\\
&=\frac{(1-\theta_1)(1-\theta_2)}{\theta_1 +\theta_2-2\theta_1\theta_2}
\Big\|\Big(\Id-(\tfrac{1}{\omega}T_2)(\omega T_1)\Big)x -\Big(\Id-(\tfrac{1}{\omega}T_2)(\omega T_1)\Big)y\Big\|^2\\
&=\frac{1-\theta}{\theta}
\Big\|\Big(\Id-(\tfrac{1}{\omega}T_2)(\omega T_1)\Big)x -\Big(\Id-(\tfrac{1}{\omega}T_2)(\omega T_1)\Big)y\Big\|^2,
\end{align}
\end{subequations}
where
\begin{equation}\label{e:200511-a}
\theta:=\frac{\theta_1+\theta_2-2\theta_1\theta_2}
{1-\theta_1\theta_2}>0.
\end{equation}
By recalling \eqref{e:T2T1} and Proposition~\ref{p:ext_avg}\ref{p:ext_avg_ori}\&\ref{p:ext_avg_ine}, we conclude that $(\tfrac{1}{\omega}T_2)(\omega T_1)$ is conically $\theta$-averaged. 

\ref{p:2averaged_theta=1}: From \eqref{e:200511-a} it follows that
\begin{equation}\label{e:2averaged_comp_1}
\theta =1
\iff
\theta_1+\theta_2-2\theta_1\theta_2 =1-\theta_1\theta_2
\iff
(1-\theta_1)(1-\theta_2) =0.
\end{equation}
Consequently, by recalling the definition of $\theta$, we conclude that $\theta=1$ if and only $\theta_1 =1$ or $\theta_2 =1$. 

\ref{p:2averaged_theta<1}: It follows from our assumption and the definition of $\theta$ that
\begin{subequations}
\begin{align}
\theta <1 
&\iff
\theta_1\theta_2 <1 \text{~~and~~} \theta_1+\theta_2-2\theta_1\theta_2 <1-\theta_1\theta_2\\
&\iff 
\theta_1\theta_2 <1 \text{~~and~~} (1-\theta_1)(1-\theta_2) >0\\
&\iff 
\theta_1 <1 \text{~and~} \theta_2 <1
\end{align}
\end{subequations}
which concludes the proof.
\end{proof}

It will be of convenience in our applications to include an equivalent reformulation of Proposition~\ref{p:2averaged}. The following result generalizes \cite[Propsition~4.44]{BC17} and also recaptures \cite[Proposition~3.12]{Gis17}, in which the composition of an averaged operator and a so-called {\it negatively averaged} operator was considered.
\begin{corollary}
\label{c:2averaged-equiv}
Let $T_1, T_2\colon X\to X$, $\omega_1,\omega_2\in\RR\smallsetminus\{0\}$, and $\theta_1, \theta_2\in \RPP$. Suppose that $\omega_1T_1$ and $\omega_2T_2$ are conically $\theta_1$-averaged and conically $\theta_2$-averaged, respectively, where either $\theta_1=\theta_2=1$ or $\theta_1\theta_2<1$.
Then $\omega_1\omega_2 T_2T_1$ and $\omega_1\omega_2 T_1T_2$ are both conically $\theta$-averaged, where
\begin{equation}
\theta :=\begin{cases}
1 &\text{if~} \theta_1 =\theta_2 =1,\\
\frac{\theta_1+\theta_2-2\theta_1\theta_2}{1-\theta_1\theta_2} &\text{if~} \theta_1\theta_2<1.
\end{cases}
\end{equation}
\end{corollary}
\begin{proof}
We note that $\omega_1\omega_2 T_2 T_1 =\omega_1(\omega_2T_2)\big(\frac{1}{\omega_1}(\omega_1T_1)\big)$ and $\omega_1\omega_2 T_1 T_2 =\omega_2(\omega_1T_1)\big(\frac{1}{\omega_2}(\omega_2T_2)\big)$. Consequently, the proof follows from  Proposition~\ref{p:2averaged}.
\end{proof}

The following theorem extends \cite[Proposition~4.9(i) and Proposition~4.46]{BC17} beyond the composition of conically averaged operators by incorporating {\em scalar multiplications} of the operators.

\begin{theorem}[composition of conically averaged operators]
\label{t:finite_comp}
Let $m\geq 2$ be an integer. For each $i\in I :=\{1,\ldots,m\}$ let $T_i\colon X\to X$ be conically $\theta_i$-averaged. Let $\{\omega_i\}_{i\in I}\subseteq \RR,\ \omega_1\omega_2\cdots\omega_m =1$. Set
\begin{equation}\label{e:finite_comp_3}
T :=\big(\omega_{m}T_{m}\big)\big(\omega_{m-1}T_{m-1}\big) \cdots \big(\omega_1T_1\big).
\end{equation}
Then 
\begin{enumerate}
\item\label{p:finite_comp-i} 
If $\max_{i\in I}\theta_i\leq 1$, then $T$ is nonexpansive.
\item\label{p:finite_comp-ii} 
If $\theta_i\neq 1$ for each $i\in I$ and
\begin{equation}\label{e:finite_comp_4}
\forall k\in \{2,\ldots,m\},\quad
\theta_{k}<1+
\frac{1}{\sum_{i=1}^{k-1} \frac{\theta_i}{1-\theta_i}},
\end{equation}
then $T$ is conically $\theta$-averaged where
\begin{equation}\label{e:finite_comp_1}
\theta:=\frac{1}{1+
\frac{1}{\sum_{i\in I}\frac{\theta_i}{1-\theta_i}}}.
\end{equation}
\item\label{p:finite_comp-iii} 
If $\max_{i\in I}\theta_i<1$, then $T$ is $\theta$-averaged where $\theta<1$ is given by \eqref{e:finite_comp_1}.
\end{enumerate}
\end{theorem}
\begin{proof}
\ref{p:finite_comp-i}: The proof follows from Proposition~\ref{p:2averaged} by induction on $m$.

\ref{p:finite_comp-ii}: We prove by induction on $m$. The case where $m=2$ is a straightforward conclusion from Proposition~\ref{p:2averaged}. Now, suppose that the statement is true for $m-1$. For each $i\in\{1,\ldots,m\}$ let $T_i$ be a conically $\theta_i$-averaged mapping such that
\begin{equation}\label{e:finite_comp_2}
\forall k\in\{2,\ldots,m\},\quad \theta_{k}<1+\frac{1}{\sum_{i=1}^{k-1} \frac{\theta_i}{1-\theta_i}}.
\end{equation}
Let $\{\omega_i\}_{i\in I}\subseteq \RR,\ \omega_1\omega_2\cdots\omega_m =1$. We prove that the operator
\begin{equation}
T=\big(\omega_{m}T_{m}\big)\big(\omega_{m-1}T_{m-1}\big) \cdots \big(\omega_1T_1\big)
\end{equation}
is conically $\theta$-averaged with $\theta$ given by \eqref{e:finite_comp_1}.

Since the statement is true for $m-1$ conically averaged operators, we have
\begin{equation}
T^*:=\big(\omega_{m}\omega_{m-1}T_{m-1}\big) \big(\omega_{m-2}T_{m-2}\big) \cdots \big(\omega_1T_1\big)
\end{equation}
is $\theta^*$-averaged where
\begin{equation}
\theta^*:=\frac{1}{1+\frac{1}{\sum_{i=1}^{m-1}\frac{\theta_i}{1-\theta_i}}}.
\end{equation}
By letting  $k=m$ in\eqref{e:finite_comp_2}, we see that
\begin{equation}
\theta_{m}<\frac{1}{\theta^*}.
\end{equation}
We now apply Proposition~\ref{p:2averaged} to two operators $T^*$ and $T_{m}$ in order to conclude that $T =(\omega_{m}T_{m})\big(\frac{1}{\omega_{m}}T^*\big)$ is conically $\theta_0$-averaged where
\begin{equation}
\theta_0 :=\frac{\theta^*+\theta_{m}-2\theta^*\theta_{m}}{ 1-\theta^*\theta_{m}}.
\end{equation}
It follows that
\begin{equation}
\frac{\theta_0}{1-\theta_0}
=\frac{\theta^*}{1-\theta^*}
+\frac{\theta_{m}}{1-\theta_{m}}
=\sum_{i=1}^{m-1}\frac{\theta_i}{1-\theta_i}
+\frac{\theta_{m}}{1-\theta_{m}}
=\sum_{i=1}^{m}\frac{\theta_i}{1-\theta_i}
\end{equation}
Consequently,
\begin{equation}
\theta_0 =\frac{1}{1+\frac{1}{\sum_{i=1}^{m}\frac{\theta_i}{1-\theta_i}}} =\theta.
\end{equation}
which concludes the proof in the case $k=m$. 

\ref{p:finite_comp-iii}: For each $i\in I$, since $\theta_i <1$, then $\frac{\theta_i}{1-\theta_i} >0$. Consequently,
\begin{equation}
\forall k\in \{2,\ldots,m\},\quad
\theta_k <1 <1+\frac{1}{\sum_{i=1}^{k-1}\frac{\theta_i}{1-\theta_i}},
\end{equation}
in particular, we see that \eqref{e:finite_comp_4} is satisfied. The conclusion now follows by employing \ref{p:finite_comp-ii}.
\end{proof}

\begin{corollary}
\label{c:finite_comp}
Let $m\geq 2$ be an integer. For each $i\in I :=\{1,\ldots,m\}$ let $T_i\colon X\to X$ and $\omega_i\in \mathbb{R}\smallsetminus \{0\}$ be such that $\omega_iT_i$ conically $\theta_i$-averaged. Set
\begin{equation}
T :=\big(\omega_{m}\omega_{m-1}\cdots \omega_1\big)\big(T_{m}T_{m-1} \cdots T_1\big).
\end{equation}
Then 
\begin{enumerate}
\item
If $\max_{i\in I}\theta_i\leq 1$, then $T$ is nonexpansive.
\item 
If $\theta_i\neq 1$ for each $i\in I$ and
\begin{equation}
\forall k\in \{2,\ldots,m\},\quad
\theta_{k}<1+
\frac{1}{\sum_{i=1}^{k-1} \frac{\theta_i}{1-\theta_i}},
\end{equation}
then $T$ is conically $\theta$-averaged where
\begin{equation}\label{e:finite_comp_1'}
\theta:=\frac{1}{1+
\frac{1}{\sum_{i\in I}\frac{\theta_i}{1-\theta_i}}}.
\end{equation}
\item
If $\max_{i\in I}\theta_i<1$, then $T$ is $\theta$-averaged where $\theta<1$ is given by \eqref{e:finite_comp_1'}.
\end{enumerate}
\end{corollary}
\begin{proof}
The proof follows by observing that
\begin{equation}
T =\Big((\omega_1\omega_2 \cdots \omega_{m-1})(\omega_mT_m)\Big)\left(\frac{1}{\omega_{m-1}}(\omega_{m-1}T_{m-1})\right)\cdots \left(\frac{1}{\omega_1}(\omega_1T_1)\right)
\end{equation}
and by applying  Theorem~\ref{t:finite_comp} to operators $\omega_1T_1$, $\omega_2T_2$, \dots, $\omega_{m-1}T_{m-1}$, $\omega_mT_m$ and scalars $\frac{1}{\omega_1}$, $\frac{1}{\omega_2}$, \dots, $\frac{1}{\omega_{m-1}}$, $\omega_1\cdots\omega_{m-1}$.
\end{proof}

We conclude this section by employing Fej\'er monotonicity in order to prove the convergence of sequences generated by an averaged/conically averaged operator. Recall that a sequence $(x_n)_\nnn$ is said to be \emph{Fej\'er monotone} with respect to a nonempty subset $C$ of $X$ if 
\begin{equation}
\forall c\in C,\ \forall\nnn, \quad \|x_{n+1}-c\|\leq \|x_n-c\|.
\end{equation}
The following result somewhat extends \cite[Theorem~5.15]{BC17}, see also \cite{BRS92, Reich79}. We include a proof for convenience. 

\begin{proposition}[Krasnosel'ski\u{\i}--Mann iterations]
\label{p:KM}
Let $T$ be a conically $\theta$-averaged operator with full domain and $\Fix T\neq \varnothing$. Let $x_0\in X$. For each $n\in\NN$ set
\begin{equation}
x_{n+1} =(1-\lambda_n)x_n +\lambda_n Tx_n,
\end{equation}
where $(\lambda_n)_\nnn$ is a sequence in $\left[0,1/\theta\right]$ such that $\sum_{n=0}^{+\infty} \lambda_n(1-\theta\lambda_n) =+\infty$.
Then
\begin{enumerate}
\item\label{p:KM_Fejer} 
$(x_n)_\nnn$ is Fej\'er monotone with respect to $\Fix T$.
\item\label{p:KM_asymp} 
$(x_n-Tx_n)_\nnn$ converges strongly to 0.
\item\label{p:KM_cvg} 
$(x_n)_\nnn$ converges weakly to a point in $\Fix T$.
\item\label{p:KM_rate}
If $\liminf_{n\to +\infty} \lambda_n(1-\theta\lambda_n) >0$, then $\|x_n-Tx_n\| =o(1/\sqrt{n})$.
\end{enumerate}
\end{proposition}
\begin{proof}
Since $T$ is conically $\theta$-averaged with full domain, then $T =(1-\theta)\Id +\theta N$ for some nonexpansive operator $N\colon X\to X$. Consequently,  $\Id-T =\theta(\Id-N)$, $\Fix T =\Fix N$, and 
\begin{equation}
\forall\nnn,\quad x_{n+1} =(1-\theta\lambda_n)x_n +\theta\lambda_n Nx_n.
\end{equation} 

\ref{p:KM_Fejer}: For all $y\in \Fix N$ and all $n\in \mathbb{N}$, by employing \eqref{e:identity} and the nonexpansiveness of $N$ we obtain
\begin{subequations}\label{e:x_+-y}
\begin{align}
\|x_{n+1}-y\|^2 &=\|(1-\theta\lambda_n)(x_n-y) +\theta\lambda_n(Nx_n-y)\|^2\\
&=(1-\theta\lambda_n)\|x_n-y\|^2 +\theta\lambda_n\|Nx_n-Ny\|^2 -\theta\lambda_n(1-\theta\lambda_n)\|x_n-Nx_n\|^2\\
&\leq \|x_n-y\|^2 -\theta\lambda_n(1-\theta\lambda_n)\|x_n-Nx_n\|^2
\end{align}
\end{subequations}
Since $\lambda_n(1-\theta\lambda_n) \geq 0$ for all $n$, $(x_n)_\nnn$ is Fej\'er monotone with respect to $\Fix N =\Fix T$.

\ref{p:KM_asymp}: By telescoping \eqref{e:x_+-y} over $n\in \mathbb{N}$, 
\begin{equation}\label{e:sum}
\theta\sum_{n=0}^{+\infty} \lambda_n(1-\theta\lambda_n)\|x_n-Nx_n\|^2 \leq \|x_0-y\|^2 <+\infty.
\end{equation} 
Since $\sum_{n=0}^{+\infty} \lambda_n(1-\theta\lambda_n) =+\infty$, it follows that
\begin{equation}\label{e:liminf}
\liminf_{n\to +\infty} \|x_n-Nx_n\| =0.
\end{equation}
Moreover, since $x_{n+1} =(1-\theta\lambda_n)x_n +\theta\lambda_n Nx_n$ and since $N$ is nonexpansive,
\begin{subequations}\label{e:decrease}
\begin{align}
\|x_{n+1}-Nx_{n+1}\| &=\|(1-\theta\lambda_n)(x_n-Nx_n) +(Nx_n-Nx_{n+1})\|\\
&\leq (1-\theta\lambda_n)\|x_n-Nx_n\| +\|Nx_n-Nx_{n+1}\|\\
&\leq (1-\theta\lambda_n)\|x_n-Nx_n\| +\|x_n-x_{n+1}\|\\
&=\|x_n-Nx_n\|.
\end{align}
\end{subequations}
We see that $\big(\|x_n-Nx_n\|\big)_\nnn$ is decreasing and bounded below by $0$, hence it converges. Consequently, \eqref{e:liminf} implies that
\begin{equation}\label{e:xn-Nxn}
x_n-Tx_n =\theta(x_n-Nx_n)\to 0\quad\text{as~} n\to +\infty.
\end{equation}

\ref{p:KM_cvg}: Let $x^*$ be a weak cluster point of $(x_n)_\nnn$, that is, there exists a subsequence $(x_{k_n})_\nnn$ such that $x_{k_n}\rightharpoonup x^*$. By combining \eqref{e:xn-Nxn} and \cite[Corollary 4.28]{BC17}, $x^*\in\Fix N$. In turn, \cite[Theorem~5.5]{BC17} implies that $(x_n)_\nnn$ converges weakly to a point in $\Fix N =\Fix T$.

\ref{p:KM_rate}: It follows from $\liminf_{n\to +\infty} \lambda_n(1-\theta\lambda_n) >0$ and \eqref{e:sum} that $\sum_{n=0}^{+\infty} \|x_n-Nx_n\|^2 <+\infty$, which, when combined with \eqref{e:decrease}, implies that
\begin{equation}
\frac{n}{2}\|x_n-Nx_n\|^2 \leq
\sum_{k=\lfloor n/2\rfloor}^{n}\|x_k-Nx_k\|^2\to 0
\quad\text{as~} n\to +\infty,
\end{equation}
where $\lfloor n/2\rfloor$ is the largest integer majorized by $n/2$. Therefore, $\|x_n-Tx_n\| =\theta\|x_n-Nx_n\| =o(1/\sqrt{n})$ as $n\to +\infty$.
\end{proof}

\begin{corollary}[convergence of averaged operators]
\label{c:avg}
Let $T$ be a $\theta$-averaged operator with full domain and $\Fix T\neq \varnothing$. Let $(x_n)_\nnn$ be a sequence generated by $T$. Then $(x_n)_\nnn$ converges weakly to a point in $\Fix T$ and the rate of asymptotic regularity of $T$ is $o(1/\sqrt{n})$, i.e., $\|x_n-Tx_n\| =o(1/\sqrt{n})$ as $n\to +\infty$.
\end{corollary}
\begin{proof}
We employ Proposition~\ref{p:KM} with $\theta <1$ and $\lambda_n =1$ for all $n\in\NN$.
\end{proof}

\section{Generalized monotonicity}
\label{s:gen_mono_coco}

We recall the following standard notations. Let $A\colon X\rightrightarrows X$. The \emph{graph} of $A$ is the set $\gra A :=\menge{(x,u)\in X\times X}{u\in Ax}$ and the \emph{inverse} of $A$, denoted by $A^{-1}$, is the operator with graph $\gra A^{-1} :=\menge{(u,x)\in X\times X}{u\in Ax}$. The \emph{resolvent} of $A$ is the mapping defined by
\begin{equation}
J_A :=(\Id+A)^{-1}
\end{equation}
and the \emph{relaxed resolvent} of $A$ with parameter $\lambda\in\RP$ is the mapping defined by
\begin{equation}
J_A^\lambda :=(1-\lambda)\Id +\lambda J_A.
\end{equation}
The mapping $R_A :=2J_A -\Id$ is the \emph{reflected resolvent} of $A$.

Let $\alpha\in \RR$. We recall that $A$ is \emph{$\alpha$-monotone} (see, for example, \cite{DP18a}) if 
\begin{equation}
\forall (x,u), (y,v)\in\gra A,\quad
\scal{x-y}{u-v}\geq \alpha\|x-y\|^2
\end{equation}
and \emph{$\alpha$-comonotone} (see, for example, \cite{BMW19}) if 
\begin{equation}
\forall (x,u), (y,v)\in\gra A,\quad
\scal{x-y}{u-v}\geq \alpha\|u-v\|^2.
\end{equation}
We say that $A$ is \emph{maximally $\alpha$-monotone} (respectively, \emph{maximally $\alpha$-comonotone}) if it is $\alpha$-monotone (respectively, $\alpha$-comonotone) and there is no $\alpha$-monotone (respectively, $\alpha$-comonotone) operator $B\colon X\rightrightarrows X$ such that $\gra A$ is properly contained in $\gra B$. 

We note that both $0$-monotonicity and $0$-comonotonicity simply mean monotonicity. If $\alpha >0$, then $\alpha$-monotonicity is actually \emph{$\alpha$-strong monotonicity} in \cite[Definition~22.1(iv)]{BC17}, while $\alpha$-comonotonicity coincides with \emph{$\alpha$-cocoercivity} \cite[Definition~4.10]{BC17}. If $\alpha <0$, then $\alpha$-monotonicity and $\alpha$-comonotonicity are respectively \emph{$\alpha$-hypomonotonicity} and \emph{$\alpha$-cohypomonotonicity} in \cite[Definition~2.2]{CP04}. Additionally, $\alpha$-monotonicity with $\alpha<0$ is also referred to as \emph{weak monotonicity} in \cite[Section~3]{DP18a}. We refer the reader to \cite{BC17,BI08} for more discussions on maximal monotonicity and some of its variants.

\begin{remark}
Several properties are immediate from the definitions.
\begin{enumerate}
\item $A$ is $\alpha$-comonotone if and only if $A^{-1}$ is $\alpha$-monotone.
\item $A$ is maximally $\alpha$-comonotone if and only if $A^{-1}$ is maximally $\alpha$-monotone.
\item If $A$ is $\alpha$-comonotone with $\alpha\geq 0$, then $A$ is also monotone.
\item If $A$ is $\alpha$-comonotone with $\alpha >0$, then $A$ is single-valued and $1/\alpha$-Lipschitz continuous.
\end{enumerate}
\end{remark}

In the case where $\alpha\geq 0$, the following characterizations of maximal $\alpha$-monotonicity and maximal $\alpha$-comonotonicity hold. 

\begin{proposition}[maximal $\alpha$-monotonicity and $\alpha$-comonotonicity]
\label{p:max-alpha}
Let $A\colon X\rightrightarrows X$ and $\alpha\in\RP$. Then 
\begin{enumerate}
\item\label{p:max-alpha_mono} 
$A$ is maximally $\alpha$-monotone if and only if $A$ is $\alpha$-monotone and maximally monotone.
\item\label{p:max-alpha_comono} 
$A$ is maximally $\alpha$-comonotone if and only if $A$ is $\alpha$-comonotone and maximally monotone.
\end{enumerate}
\end{proposition}
\begin{proof}
\ref{p:max-alpha_mono}: See~\cite[Proposition~3.5(i)]{DP18a}.

\ref{p:max-alpha_comono}: $A$ is maximally $\alpha$-comonotone $\iff$
$A^{-1}$ is maximally $\alpha$-monotone $\iff$ $A^{-1}$ is $\alpha$-monotone and maximally monotone (by \ref{p:max-alpha_mono}) $\iff$ $A$ is $\alpha$-comonotone and maximally monotone. 
\end{proof}

We now collect several useful properties of relaxed resolvents of $\alpha$-monotone and $\alpha$-comonotone operators. Some parts of the following results are available in \cite{BC17,BMW19,DP18a,DP18b}. For convenience, these results are included here as well as their proofs. In particular, we will show that if an operator is either $\alpha$-monotone or $\alpha$-comonotone, then its relaxed resolvent is, to a certain extent, related to a conically averaged operator. These results play a crucial role in convergence analysis of several iterative algorithms. We begin our discussion with the following auxiliary properties.

\begin{proposition}
\label{p:cocoer-avg}
Let $A\colon X\to X$ and $\alpha,\lambda\in \RPP$. Then the following assertions are equivalent:
\begin{enumerate}
\item\label{p:cocoer-avg_coco} 
$A$ is $\alpha$-comonotone (i.e., $\alpha$-cocoercive).
\item\label{p:cocoer-avg_fn}
$\alpha A$ is firmly nonexpansive.
\item\label{p:cocoer-avg_avg}
$\Id-\lambda A$ is conically $\frac{\lambda}{2\alpha}$-averaged. 
\end{enumerate}
\end{proposition}
\begin{proof}
The equivalence between \ref{p:cocoer-avg_coco} and \ref{p:cocoer-avg_fn} follows, e.g., from \cite[Remark~4.34(iv)]{BC17}. The equivalence between \ref{p:cocoer-avg_fn} and \ref{p:cocoer-avg_avg} follows from Proposition~\ref{p:ext_averaged2} by observing that $\Id-\lambda A =\Id-\frac{\lambda}{\alpha}(\alpha A)$.
\end{proof}

\begin{proposition}[single-valuedness and full domain]
\label{p:resol-mono}
Let $A\colon X\rightrightarrows X$ be $\alpha$-monotone and let $\gamma\in\RPP$ such that $1+\gamma\alpha>0$. Then 
\begin{enumerate}
\item\label{p:resol-mono_single} 
$J_{\gamma A}$ is (at most) single-valued. 
\item 
$\dom J_{\gamma A}=X$ if and only if $A$ is maximally $\alpha$-monotone.
\end{enumerate}
\end{proposition}
\begin{proof}
See \cite[Proposition~3.4]{DP18a}.
\end{proof}

\begin{proposition}[relaxed resolvents of $\alpha$-monotone operators]
\label{p:rresol-mono}
Let $A\colon X\rightrightarrows X$ be $\alpha$-monotone and let $\gamma\in \RPP$ be such that $1+\gamma\alpha >0$. Set $R :=(1-\lambda)\Id +\lambda J_{\gamma A}$ where $\lambda\in \left]1,+\infty\right[$. 
Then 
\begin{enumerate}
\item\label{p:rresol-mono_J}
$J_{\gamma A}$ is $(1+\gamma\alpha)$-comonotone. Consequently, $(1+\gamma\alpha)J_{\gamma A}$ is firmly nonexpansive.
\item\label{p:rresol-mono_R} 
$\frac{1}{1-\lambda}R$ is conically $\frac{\lambda}{2(\lambda-1)(1+\gamma\alpha)}$-averaged.
\end{enumerate} 
\end{proposition}
\begin{proof}
\ref{p:rresol-mono_J}: By Proposition~\ref{p:resol-mono}\ref{p:resol-mono_single} $J_{\gamma A}$ is single-valued, and, by \cite[Lemma~3.3]{DP18a}, it is $(1+\gamma\alpha)$-cocoercive, i.e., $(1+\gamma\alpha)$-comonotone. Since $1+\gamma\alpha >0$, Proposition~\ref{p:cocoer-avg} implies that $(1+\gamma\alpha)J_{\gamma A}$ is firmly nonexpansive.

\ref{p:rresol-mono_R}: Since $(1+\gamma\alpha)J_{\gamma A}$ is firmly nonexpansive, Proposition~\ref{p:ext_averaged2} implies that 
\begin{equation}
\frac{1}{1-\lambda}R =\Id +\frac{\lambda}{1-\lambda}J_{\gamma A} =\Id -\frac{\lambda}{(\lambda-1)(1+\gamma\alpha)}(1+\gamma\alpha)J_{\gamma A}
\end{equation}
is conically $\frac{\lambda}{2(\lambda-1)(1+\gamma\alpha)}$-averaged.
\end{proof}

\begin{lemma}[resolvents of $\alpha$-comonotone operators]
\label{l:resol-cocoer}
Let $A\colon X\rightrightarrows X$ and $\gamma\in \RPP$. Then $A$ is $\alpha$-comonotone if and only if for all $(x,a), (y,b)\in \gra J_{\gamma A}$, 
\begin{equation}\label{e:resol-cocoer1}
(\gamma+2\alpha)\scal{x-y}{a-b}\geq \alpha\|x-y\|^2 +(\gamma+\alpha)\|a-b\|^2.
\end{equation}
Consequently, if $A$ is $\alpha$-comonotone and $J_{\gamma A}$ is single-valued, then, for all $x,y\in\dom J_{\gamma A}$, 
\begin{equation}
\label{e:Jbi}
(\gamma+2\alpha)\scal{x-y}{J_{\gamma A}x-J_{\gamma A}y}\geq \alpha\|x-y\|^2 +(\gamma+\alpha)\|J_{\gamma A}x-J_{\gamma A}y\|^2.
\end{equation}
\end{lemma}
\begin{proof}
Let $(a,u), (b,v)\in \gra A$ and set $x :=a +\gamma u$ and $y :=b +\gamma v$. Then
\begin{subequations}
\begin{align}
& \scal{a-b}{u-v}\geq \alpha\|u-v\|^2\\
\iff\ & \gamma\scal{a-b}{\gamma u-\gamma v}\geq \alpha\|\gamma u-\gamma v\|^2\\
\iff\ & \gamma\scal{a-b}{(x-y)-(a-b)}\geq \alpha\|(x-y)-(a-b)\|^2\\
\iff\ & \gamma\scal{x-y}{a-b} -\gamma\|a-b\|^2\geq \alpha\left(\|x-y\|^2 +\|a-b\|^2 -2\scal{x-y}{a-b}\right)\\
\iff\ & (\gamma+2\alpha)\scal{x-y}{a-b}\geq \alpha\|x-y\|^2 +(\gamma+\alpha)\|a-b\|^2
\end{align}
\end{subequations}
which completes the proof.
\end{proof}

\begin{proposition}[single-valuedness and full domain]
\label{p:resol-cocoer}
Let $A\colon X\rightrightarrows X$ be $\alpha$-comonotone and let $\gamma\in \RPP$ be such that $\gamma+\alpha >0$. Then 
\begin{enumerate}
\item\label{p:resol-cocoer_single} 
$J_{\gamma A}$ is (at most) single-valued.
\item\label{p:resol-cocoer_dom} 
$\dom J_{\gamma A} =X$ if and only if $A$ is maximally $\alpha$-comonotone.
\end{enumerate}
\end{proposition}
\begin{proof}
\ref{p:resol-cocoer_single}: This follows from \eqref{e:resol-cocoer1} in Lemma~\ref{l:resol-cocoer} and the fact that $\gamma+\alpha>0$.

\ref{p:resol-cocoer_dom}: Since $A$ is $\alpha$-comonotone, $A' :=A^{-1}-\alpha\Id$ is monotone. For $x\in X$, 
\begin{subequations}
\begin{align}
a\in (\Id-J_{\gamma A})x 
&\iff x\in (x-a) +\gamma A(x-a) \\
&\iff x-a\in A^{-1}\left(\frac{a}{\gamma}\right) =(\alpha\Id+A')\left(\frac{a}{\gamma}\right)\\
&\iff x\in (\gamma+\alpha)\left(\Id+\frac{1}{\gamma+\alpha}A'\right)\left(\frac{a}{\gamma}\right) \\
&\iff a\in \gamma\left(\Id+\frac{1}{\gamma+\alpha}A'\right)^{-1} \left(\frac{x}{\gamma+\alpha}\right).
\end{align}
\end{subequations}
Therefore,
\begin{equation}
\Id -J_{\gamma A} =\gamma J_{\frac{1}{\gamma+\alpha}A'} \left(\frac{1}{\gamma+\alpha}\Id\right).
\end{equation}
We conclude that $\dom J_{\gamma A} =X$ if and only if $\dom J_{\frac{1}{\gamma+\alpha}A'}=X$, which is equivalent to  $A'$ being maximally monotone (see \cite[Theorem~21.1 and Proposition~20.22]{BC17}). The proof follows by observing that the maximal monotonicity of $A'$ is equivalent to the maximal $\alpha$-comonotonicity of $A$.
\end{proof}

The following result is an extension of \cite[Proposition~23.14]{BC17}.
\begin{proposition}[relaxed resolvents of $\alpha$-comonotone operators]
\label{p:rresol-cocoer}
Let $A\colon X\rightrightarrows X$ be $\alpha$-comonotone and let $\gamma\in \RPP$ be such that $\gamma+\alpha >0$. Set $R :=(1-\lambda)\Id +\lambda J_{\gamma A}$ with $\lambda\in\RPP$. Then
\begin{enumerate}
\item\label{p:rresol-cocoer_J} 
$J_{\gamma A}$ is conically $\frac{\gamma}{2(\gamma+\alpha)}$-averaged.
\item\label{p:rresol-cocoer_R}  
$R$ is conically $\frac{\lambda\gamma}{2(\gamma+\alpha)}$-averaged.
\end{enumerate}
\end{proposition}
\begin{proof}
Proposition~\ref{p:resol-cocoer}\ref{p:resol-cocoer_single} implies that $J_{\gamma A}$ is single-valued. Consequently, Lemma~\ref{l:resol-cocoer} implies that, for all $x,y\in\dom J_{\gamma A}$, 
\begin{equation}
(\gamma+2\alpha)\scal{x-y}{J_{\gamma A}x-J_{\gamma A}y}\geq \alpha\|x-y\|^2 +(\gamma+\alpha)\|J_{\gamma A}x-J_{\gamma A}y\|^2
\end{equation} 
which is equivalent to 
\begin{equation}
2\left(1-\frac{\gamma}{2(\gamma+\alpha)}\right)\scal{x-y}{J_{\gamma A}x-J_{\gamma A}y}\geq \left(1-\frac{\gamma}{\gamma+\alpha}\right)\|x-y\|^2 +\|J_{\gamma A}x-J_{\gamma A}y\|^2.
\end{equation}
Proposition~\ref{p:ext_avg}\ref{p:ext_avg_ori}\&\ref{p:ext_avg_ine'} now implies that $J_{\gamma A}$ is conically $\frac{\gamma}{2(\gamma+\alpha)}$-averaged. In turn, Proposition~\ref{p:ext_avg}\ref{p:ext_avg_ori}\&\ref{p:ext_avg_rlx} implies that $R$ is conically $\frac{\lambda\gamma}{2(\gamma+\alpha)}$-averaged.
\end{proof}

We will now address the connection between conical averagedness and several fixed point algorithms including the forward-backward algorithm and the adaptive Douglas--Rachford algorithm.

\section{Relaxed forward-backward algorithm}
\label{s:rfb}

Let $A\colon X\rightrightarrows X$, $B\colon X\to X$. We consider the problem
\begin{equation}\label{e:sumprob}
\text{find $x\in X$ such that}\quad
0\in Ax+Bx.
\end{equation}
Let $\gamma\in \RPP$ and $\kappa\in \RPP$. Set $x_0\in X$. The \emph{relaxed forward-backward (rFB) algorithm} for problem \eqref{e:sumprob} generates a sequence $(x_n)_\nnn$ via 
\begin{equation}\label{e:Tfb}
\forall\nnn,\quad x_{n+1}\in T_{\rm FB}x_n\quad \text{where\ ~} T_{\rm FB} :=(1-\kappa)\Id +\kappa J_{\gamma A}(\Id-\gamma B).
\end{equation}  
In the case where $\kappa =1$, the rFB algorithm is the well-studied \emph{forward-backward algorithm}, see, e.g., \cite[Section~26.5]{BC17}.

When one considers an iterative fixed point algorithm in order to solve a problem, the relations between the fixed points of that algorithm and the solutions of the problem under consideration is crucial. The following well-known result, see, e.g., \cite[Proposition~26.1(iv)(a)]{BC17}, asserts that the fixed points of the forward-backward algorithm are, in fact, solutions of problem \eqref{e:sumprob}.  For completeness, we include a proof as well. 

\begin{lemma}
\label{l:FixFB}
With the settings of~\eqref{e:sumprob}--\eqref{e:Tfb},
\begin{equation}
\Fix T_{\rm FB} =\Fix\big(J_{\gamma A}(\Id-\gamma B)\big) =\zer(A+B).
\end{equation}
\end{lemma}
\begin{proof}
We note that $\Id-T_{\rm FB} =\kappa\big(\Id-J_{\gamma A}(\Id-\gamma B)\big)$. Consequently, $\Fix T_{\rm FB} =\Fix\big(J_{\gamma A}(\Id-\gamma B)\big)$. It follows that
$x\in\Fix\big(J_{\gamma A}(\Id-\gamma B)\big)$
$\iff$ $x\in J_{\gamma A}(x-\gamma Bx)$
$\iff$ $x-\gamma Bx \in x+\gamma Ax$
$\iff$ $0\in Ax+Bx$.
\end{proof}

In the case where $B=0$, problem \eqref{e:sumprob} reduces to finding a zero of the operator $A\colon X\rightrightarrows X$, i.e., 
\begin{equation}
\text{find $x\in X$ such that}\quad
0\in Ax
\end{equation}
and the corresponding rFB algorithm reduces to the \emph{relaxed proximal point algorithm} of the form
\begin{equation}\label{e:prox_alg}
\forall\nnn,\quad x_{n+1}\in T_{\rm PP}x_n\quad \text{where\ ~} T_{\rm PP} :=(1-\kappa)\Id +\kappa J_{\gamma A}.
\end{equation}
In this case, Lemma~\ref{l:FixFB} implies that 
\begin{equation}\label{e:FixPP}
\Fix T_{\rm PP} =\Fix J_{\gamma A} =\zer A.
\end{equation}

The following results are the main results of this section. We provide the averagedness of $T_{\rm PP}$ and $T_{\rm FB}$ as well as the convergence of the corresponding algorithms in cases where $A$ is not necessarily monotone. Classical results for monotone operators can be found, for example, in~\cite[Example~23.40 and Proposition~26.1(iv)(d)]{BC17}.

\begin{theorem}[relaxed proximal point algorithm]
\label{t:pp}
Suppose that $A$ is maximally $\alpha$-comonotone with $\alpha\in \RR$ and that $\gamma >\max\{0,-\alpha\}$. 
Then $T_{\rm PP}$ is conically $\frac{\kappa}{\kappa^*}$-averaged and has full domain, where $\kappa^* :=\frac{2(\gamma+\alpha)}{\gamma}$. Moreover, if $\zer A\neq \varnothing$ and $\kappa <\kappa^*$, then every sequence $(x_n)_\nnn$ generated by $T_{\rm PP}$ converges weakly to a point in $\zer A$ and the rate of asymptotic regularity of $T_{\rm PP}$ is $o(1/\sqrt{n})$.
\end{theorem}
\begin{proof}
As $\gamma+\alpha>0$, employing Proposition~\ref{p:resol-cocoer}, we see that $J_{\gamma A}$ and hence $T_{\rm PP}$ are single-valued and have full domain. By Proposition~\ref{p:rresol-cocoer}\ref{p:rresol-cocoer_R}, $T_{\rm PP}$ is conically $\theta$-averaged, where
\begin{equation}
\theta :=\frac{\kappa\gamma}{2(\gamma+\alpha)} =\frac{\kappa}{\kappa^*}.
\end{equation}
The proof follows from Corollary~\ref{c:avg} and \eqref{e:FixPP}.
\end{proof}

\begin{theorem}[relaxed forward-backward algorithm]
\label{t:FB}
Suppose that $A$ is maximally $\alpha$-comonotone with $\alpha\in \RR$, that $B$ is $\beta$-comonotone with $\beta\in \RPP$ (i.e., $\beta$-cocoercive), and that either
\begin{enumerate}
\item\label{t:FB_0} 
$\alpha +\beta =0$ and $\gamma =2\beta$; or
\item\label{t:FB_strong} 
$\alpha+\beta >0$ and $\max\{0,2\beta-2\sqrt{\beta(\alpha+\beta)}\} <\gamma <2\beta+2\sqrt{\beta(\alpha+\beta)}$.
\end{enumerate}
Then $T_{\rm FB}$ is conically $\frac{\kappa}{\kappa^*}$-averaged and has full domain, where 
\begin{equation}
\kappa^* :=\begin{cases}
1 &\text{if~} \alpha +\beta =0,\\
\frac{4(\gamma+\alpha)\beta -\gamma^2}{2\gamma(\alpha+\beta)} &\text{if~} \alpha+\beta >0.
\end{cases}
\end{equation}
Moreover, if $\zer(A+B)\neq \varnothing$ and $\kappa <\kappa^*$, then every sequence $(x_n)_\nnn$ generated by $T_{\rm FB}$ converges weakly to a point in $\zer(A+B)$ and the rate of asymptotic regularity of $T_{\rm FB}$ is $o(1/\sqrt{n})$.
\end{theorem}
\begin{proof}
On the one hand, \ref{t:FB_0} implies that $\gamma+\alpha =2\beta-\beta =\beta >0$ and $\gamma =2(\gamma+\alpha) =2\beta$. On the other hand, \ref{t:FB_strong} is equivalent to $\gamma^2 <4(\gamma+\alpha)\beta$, which implies that $\gamma+\alpha >0$ (since $\beta >0$) and $\kappa^* >0$ (since $\gamma >0$). Now, since $A$ is maximally $\alpha$-comonotone, Proposition~\ref{p:resol-cocoer} implies that $J_{\gamma A}$ is single-valued and has full domain, and so does $T_{\rm FB}$. By Proposition~\ref{p:rresol-cocoer}\ref{p:rresol-cocoer_J}, $J_{\gamma A}$ is conically $\theta_1$-averaged, where
\begin{equation}
\theta_1 :=\frac{\gamma}{2(\gamma+\alpha)}.
\end{equation}
Next, since $B$ is $\beta$-comonotone, Proposition~\ref{p:cocoer-avg} implies that $\Id-\gamma B$ is conically $\theta_2$-averaged, where
\begin{equation}
\theta_2 :=\frac{\gamma}{2\beta}.
\end{equation} 
We observe that if \ref{t:FB_0} holds, then $\theta_1 =\theta_2 =1$ and $\frac{1}{\kappa^*} =1$; if \ref{t:FB_strong} holds, then $\theta_1\theta_2 <1$ and
\begin{equation}
\frac{1}{\kappa^*} =\frac{2\gamma(\alpha+\beta)}{4(\gamma+\alpha)\beta -\gamma^2} =\frac{\theta_1+\theta_2-2\theta_2\theta_2}{1-\theta_1\theta_2}.
\end{equation}   
In view of Proposition~\ref{p:2averaged}, $J_{\gamma A}(\Id-\gamma B)$ is conically $\frac{1}{\kappa^*}$-averaged, and by Proposition~\ref{p:ext_avg}\ref{p:ext_avg_ori}\&\ref{p:ext_avg_rlx}, $T_{\rm FB}$ is conically $\frac{\kappa}{\kappa^*}$-averaged. The proof then follows from Corollary~\ref{c:avg} and Lemma~\ref{l:FixFB}.   
\end{proof}

\begin{corollary}
\label{c:mono-cocoer}
Suppose that $A$ is maximally monotone, that $B$ is $\beta$-comonotone with $\beta\in \RPP$ (i.e., $\beta$-cocoercive), and that $\gamma\in \left]0,4\beta\right[$.
Then $T_{\rm FB}$ is conically $\frac{\kappa}{\kappa^*}$-averaged and has full domain, where $\kappa^*:=\frac{4\beta-\gamma}{2\beta}$. Moreover, if $\zer(A+B)\neq \varnothing$ and $\kappa <\kappa^*$, then every sequence $(x_n)_\nnn$ generated by $T_{\rm FB}$ converges weakly to a point in $\zer(A+B)$ and the rate of asymptotic regularity of $T_{\rm FB}$ is $o(1/\sqrt{n})$.
\end{corollary}
\begin{proof}
Since $A$ is maximally $0$-comonotone, we apply Theorem~\ref{t:FB}\ref{t:FB_strong} with $\alpha =0$.
\end{proof}

\begin{remark}[range of parameter $\gamma$]
We recall that classical convergence analysis for the forward-backward algorithm requires that $\gamma\in \left]0,2\beta\right[$, see, for example, \cite[Proposition~26.1(iv)(d) and Theorems~26.14(i)]{BC17}. Corollary~\ref{c:mono-cocoer} improves upon that by only requiring $\gamma\in \left]0,4\beta\right[$.
\end{remark}

\section{Adaptive Douglas--Rachford algorithm}
\label{s:adr}

We focus on problem \eqref{e:sumprob} where $A\colon X\rightrightarrows X$ and $B\colon X\rightrightarrows X$. We set $(\gamma,\delta)\in\RPP^2$ and $(\lambda,\mu,\kappa)\in\RPP^3$. The \emph{adaptive DR operator}, introduced and studied in \cite{DP18a}, is defined by
\begin{equation}\label{e:aDR}
T :=T_{A,B} := (1-\kappa)\Id+ \kappa R_2R_1
\end{equation}
where 
\begin{equation}
R_1 :=(1-\lambda)\Id+\lambda J_{\gamma A}
\quad\text{and}\quad
R_2 :=(1-\mu)\Id+\mu J_{\delta B}.
\end{equation}
Set $x_0\in X$. Then the \emph{adaptive DR (aDR) algorithm} for problem \eqref{e:sumprob} generates a sequence $(x_n)_\nnn$, also called a \emph{DR sequence}, by letting
\begin{equation}\label{e:dr_seq}
\forall\nnn,\quad x_{n+1}\in Tx_n.
\end{equation}
Naturally, we refer to the case where $\delta=\gamma>0$ and $\lambda=\mu=2$ as the \emph{classical DR algorithm} (or simply DR).

Unlike in the case of the forward-backward counterpart, the fixed points of the adaptive DR algorithm, in general, do not directly solve \eqref{e:sumprob}. Nevertheless, by choosing compatible parameters, we show that the images of the fixed points under the resolvent are, in fact, solutions. To this end, similarly to \cite[Section~4]{DP18a}, we assume that
\begin{equation}
(\lambda-1)(\mu-1)=1
\text{~~and~~}
\delta=\gamma(\lambda-1),
\end{equation}
equivalently,
\begin{equation}\label{e:CQ}
\lambda =1+\frac{\delta}{\gamma} 
\text{~~and~~} 
\mu =1+\frac{\gamma}{\delta}
\end{equation}
which clearly holds for the classical DR algorithm. Our settings are justified by the following fact.

\begin{fact}[fixed points of the aDR operator]
\label{f:fix}
Suppose that \eqref{e:CQ} holds. Then $\Fix T\neq \varnothing$ if and only if $\zer(A+B)\neq\varnothing$. Moreover, if $J_{\gamma A}$ is single-valued, then
\begin{equation}
J_{\gamma A}(\Fix T)=\zer(A+B).
\end{equation}
\end{fact}
\begin{proof}
See \cite[Lemma~4.1(iii)]{DP18a}.
\end{proof}

In view of Fact~\ref{f:fix} we focus on the convergence of the adaptive DR algorithm to a fixed point of the operator $T$ under condition~\eqref{e:CQ}. In turn, such convergence can be guaranteed by the averagedness as we show in the following result.

\begin{proposition}[convergence of the aDR algorithm via averagedness]
\label{p:cvg}
Suppose that $T$ is $\theta$-averaged and has full domain, that $\zer(A+B)\neq\varnothing$, and that \eqref{e:CQ} holds. Then the rate of assymptotic regularity of $T$ is $o(1/\sqrt{n})$ and every sequence $(x_n)_\nnn$ generated by $T$ converges weakly to a point $\overline{x}\in \Fix T$. Moreover, if $J_{\gamma A}$ is single-valued, then $J_{\gamma A}\overline{x}\in \zer(A+B)$. 
\end{proposition}
\begin{proof}
Since $\zer(A+B)\neq\varnothing$, Fact~\ref{f:fix} implies that $\Fix T\neq \varnothing$. Now, by Corollary~\ref{c:avg}, the rate of asymptotic regularity of $T$ is $o(1/\sqrt{n})$ and every sequence $(x_n)_\nnn$ generated by $T$ converges weakly to a point $\overline{x}\in \Fix T$. Moreover, if $J_{\gamma A}$ is single-valued, then Fact~\ref{f:fix} implies that $J_{\gamma A}\overline{x}\in \zer(A+B)$. 
\end{proof}

Motivated by these observations, we focus on the averagedness of the adaptive DR operator. To this end, we look for compatible parameters $\gamma,\delta,\lambda,\mu,\kappa$. In fact, it is enough to determine only $\gamma,\delta,\kappa>0$ which, in turn, determine $\lambda,\mu$ via \eqref{e:CQ}.

\subsection{The case of $\alpha$-comonotone and $\beta$-comonotone operators}
\label{ss:ABcocoer}

We consider the adaptive DR operators for two comonotone operators. In particular, we derive convergence by employing conical averagedness. To this end, we will make use of the following lemma.

\begin{lemma}[existence of parameters]
\label{l:exist'}
Let $\alpha,\beta\in \RR$ be such that $\alpha+\beta\geq 0$, and let $\gamma, \delta\in \RPP$. Set 
\begin{equation}\label{e:gamma0}
\gamma_0 :=\begin{cases}
0 &\text{if~} \alpha\geq 0,\\
2\beta -2\sqrt{\beta(\alpha+\beta)} &\text{if~} \alpha <0.
\end{cases}
\end{equation}
and set $\Delta :=(\gamma+\alpha)(\alpha+\beta)$. Then $\gamma_0\geq \max\{0,-\alpha\}$ and the following assertions are equivalent:
\begin{enumerate}
\item\label{l:exist'_i} 
$(\gamma+\delta)^2 \leq 4(\gamma+\alpha)(\delta+\beta)$.
\item\label{l:exist'_ii}
$\gamma >\gamma_0$ and 
$\gamma+2\alpha -2\sqrt{\Delta} \leq \delta \leq \gamma+2\alpha +2\sqrt{\Delta}$.
\end{enumerate}
Consequently, if $\alpha+\beta =0$, then \ref{l:exist'_i} and \ref{l:exist'_ii} are equivalent to $\gamma >\max\{0,-2\alpha\}$ and $\delta =\gamma+2\gamma$; if $\alpha+\beta >0$, then there always exist $\gamma, \delta\in \RPP$ such that
\begin{equation}
\gamma >\gamma_0 \quad\text{and}\quad
\max\{0,\gamma+2\alpha -2\sqrt{\Delta}\} <\delta <\gamma+2\alpha +2\sqrt{\Delta},
\end{equation}
in which case, all inequalities in \ref{l:exist'_i} and \ref{l:exist'_ii} are strict. Moreover, if \ref{l:exist'_i} or \ref{l:exist'_ii} holds, then $\gamma+\alpha >0$ and $\delta+\beta >0$.
\end{lemma}
\begin{proof}
If $\alpha\geq 0$, then $\gamma_0 =0 =\max\{0, -\alpha\}$. If $\alpha <0$, then $\beta >\alpha+\beta \geq 0$ and
\begin{equation}\label{e:gamma0'}
\gamma_0 =2\beta -2\sqrt{\beta(\alpha+\beta)} =(\sqrt{\beta} -\sqrt{\alpha+\beta})^2 -\alpha\geq -\alpha =\max\{0, -\alpha\}.
\end{equation}

Next, it is clear that
\begin{subequations}\label{e:exist'_equiv}
\begin{align}
&(\gamma+\delta)^2\leq 4(\gamma+\alpha)(\delta+\beta)\\
\iff\ &\delta^2 -2(\gamma+2\alpha)\delta +\gamma^2-4(\gamma+\alpha)\beta \leq 0\\
\iff\ &\Delta =(\gamma+\alpha)(\alpha+\beta)\geq 0
\text{~~and~~}
\gamma+2\alpha -2\sqrt{\Delta} \leq \delta \leq\gamma+2\alpha +2\sqrt{\Delta}.
\end{align}
\end{subequations}
In addition, since $\gamma_0\geq \{0, -\alpha\}$ and $\alpha+\beta\geq 0$, we have $\Delta\geq 0$ as soon as $\gamma >\gamma_0$. By combining this with \eqref{e:exist'_equiv}, we see that \ref{l:exist'_ii} implies \ref{l:exist'_i}. 

We now assume \ref{l:exist'_i}. Then, by \eqref{e:exist'_equiv}, $\Delta\geq 0$ and $\gamma+2\alpha -2\sqrt{\Delta} \leq \delta \leq\gamma+2\alpha +2\sqrt{\Delta}$. Since $\delta >0$, it follows that $\gamma+2\alpha +2\sqrt{\Delta} >0$. To show \ref{l:exist'_ii}, it suffices to show that $\gamma >\gamma_0$. We distinguish between the following cases.

\emph{Case~1:} $\alpha+\beta =0$. Then $\gamma_0 =\max\{0,-2\alpha\}$, $\Delta =0$, and 
\begin{equation}
\gamma+2\alpha +2\sqrt{\Delta} >0 \iff \gamma+2\alpha >0, 
\end{equation}
which implies that $\gamma >\gamma_0 =\max\{0,-2\alpha\}$.

\emph{Case 2:} $\alpha+\beta >0$. Since $\Delta\geq 0$, then $\gamma+\alpha\geq 0$. If $\alpha\geq 0$, then $\gamma_0 =0$, so $\gamma >\gamma_0$. Now, assume that $\alpha <0$. Then $\beta >\alpha+\beta \geq 0$ and 
\begin{subequations}\label{e:gamma0''}
\begin{align}
\gamma+2\alpha +2\sqrt{\Delta} >0
&\iff
(\sqrt{\gamma+\alpha} +\sqrt{\alpha+\beta})^2 >\beta\\
&\iff
\sqrt{\gamma+\alpha} >\sqrt{\beta} -\sqrt{\alpha+\beta}\\
&\iff \gamma+\alpha >(\sqrt{\beta} -\sqrt{\alpha+\beta})^2\\ 
&\iff \gamma >\gamma_0 =2\beta -2\sqrt{\beta(\alpha+\beta)}.
\end{align}
\end{subequations}
This completes the proof of the equivalence between \ref{l:exist'_i} and \ref{l:exist'_ii}. On the other hand, if we set $\gamma >\gamma_0$, then by \eqref{e:gamma0''}, there exists $\delta$ such that $\max\{0,\gamma+2\alpha-2\sqrt{\Delta}\} <\delta <\gamma+2\alpha+2\sqrt{\Delta}$, which leads to strict inequalities in \ref{l:exist'_i} and \ref{l:exist'_ii}.

Finally, since \ref{l:exist'_i} and \ref{l:exist'_ii} are equivalent, if one of them holds, then $\gamma >\gamma_0\geq \max\{0,-\alpha\}$, so $\gamma+\alpha >0$, which, when combined with \ref{l:exist'_i}, implies $\delta+\beta >0$. 
\end{proof}

\begin{theorem}[aDR for $\alpha$-comonotone and $\beta$-comonotone operators]
\label{t:2cocoer}
Suppose that $A$ and $B$ are maximally $\alpha$-comonotone and $\beta$-comonotone, respectively. Set $\lambda,\mu$ as in \eqref{e:CQ}, and suppose that either
\begin{enumerate}
\item\label{t:2cocoer_0} 
$\alpha+\beta =0$, $\gamma >\max\{0,-2\alpha\}$, $\delta =\gamma+2\alpha$, and $\kappa^* :=1$; or
\item\label{t:2cocoer_strong}
$\alpha+\beta >0$ and $\kappa^*:=\frac{4(\gamma+\alpha)(\delta+\beta)-(\gamma+\delta)^2}{2(\gamma+\delta)(\alpha+\beta)} >0$. 
\end{enumerate} 
Then the adaptive DR operators $T_{A,B}$ and $T_{B,A}$ are conically $\frac{\kappa}{\kappa^*}$-averaged and have full domain. Moreover, if $\zer(A+B)\neq \varnothing$ and $\kappa <\kappa^*$, then, for any $(T,C)\in \{(T_{A,B},\gamma A), (T_{B,A},\delta B)\}$, every sequence $(x_n)_\nnn$ generated by $T$ converges weakly to a point $\overline{x}\in \Fix T$ with $J_{C}\overline{x}\in \zer(A+B)$ and the rate of asymptotic regularity of $T$ is $o(1/\sqrt{n})$.
\end{theorem}
\begin{proof}
We first observe that in \ref{t:2cocoer_0}, the existence of $\gamma, \delta\in \RPP$ is clear, while in \ref{t:2cocoer_strong} it follows from Lemma~\ref{l:exist'} which asserts the existence of $\gamma, \delta\in \RPP$ such that
\begin{equation}
(\gamma+\delta)^2 <4(\gamma+\alpha)(\delta+\beta).
\end{equation}
Also, Lemma~\ref{l:exist'} implies that $\gamma+\alpha >0$ and $\delta+\beta >0$ in both cases \ref{t:2cocoer_0} and \ref{t:2cocoer_strong}. By employing Proposition~\ref{p:resol-cocoer}, $J_{\gamma A}$, $J_{\delta B}$ and, hence, $T_{A,B}$ and $T_{B,A}$ are single-valued and have full domain. 

Next, Proposition~\ref{p:rresol-cocoer} implies that $R_1$ and $R_2$ are conically $\theta_1$-averaged and $\theta_2$-averaged, respectively, where
\begin{equation}
\theta_1 :=\frac{\lambda\gamma}{2(\gamma+\alpha)} =\frac{\gamma+\delta}{2(\gamma+\alpha)} \text{~~and~~} \theta_2 :=\frac{\mu\delta}{2(\delta+\beta)} =\frac{\gamma+\delta}{2(\delta+\beta)}.
\end{equation}
Now, if \ref{t:2cocoer_0} holds, then $\theta_1 =\theta_2 =1$ and $\frac{1}{\kappa^*} =1$; if \ref{t:2cocoer_strong} holds, then $\theta_1\theta_2 <1$ and 
\begin{equation}
\frac{1}{\kappa^*} =\frac{2(\gamma+\delta)(\alpha+\beta)}{4(\gamma+\alpha)(\delta+\beta)-(\gamma+\delta)^2} =\frac{\theta_1+\theta_2-2\theta_1\theta_2}{1-\theta_1\theta_2}.
\end{equation}
Thus, Corollary~\ref{c:2averaged-equiv} implies that $R_1R_2$ and $R_2R_1$ are conically $\frac{1}{\kappa^*}$-averaged. By invoking Proposition~\ref{p:ext_avg}\ref{p:ext_avg_ori}\&\ref{p:ext_avg_rlx}, we conclude that $T_{A,B}$ and $T_{B,A}$ are conically $\frac{\kappa}{\kappa^*}$-averaged. Finally, due to Proposition~\ref{p:cvg}, the proof is complete.
\end{proof}

\begin{corollary}[DR for $\alpha$-comonotone and $\beta$-comonotone operators]
Suppose that $A$ and $B$ are maximally $\alpha$-comonotone and $\beta$-comonotone, respectively, that $\gamma =\delta\in \RPP$ and $\lambda =\mu =2$. Suppose further that either
\begin{enumerate}
\item 
$\alpha = \beta =0$ and $\kappa^* :=1$; or
\item 
$\alpha+\beta >0$ and $\kappa^* :=\gamma+\frac{\alpha\beta}{\alpha+\beta} >0$.
\end{enumerate}
Then the adaptive DR operators $T_{A,B}$ and $T_{B,A}$ are conically $\frac{\kappa}{\kappa^*}$-averaged and have full domain. Moreover, if $\zer(A+B)\neq \varnothing$ and $\kappa <\kappa^*$, then, for any $(T,C)\in \{(T_{A,B},\gamma A), (T_{B,A},\delta B)\}$, every sequence $(x_n)_\nnn$ generated by $T$ converges weakly to a point $\overline{x}\in \Fix T$ with $J_{C}\overline{x}\in \zer(A+B)$ and the rate of asymptotic regularity of $T$ is $o(1/\sqrt{n})$.
\end{corollary}
\begin{proof}
We invoke Theorem~\ref{t:2cocoer} with $\gamma =\delta$ and $\lambda =\mu =2$.
\end{proof}

\subsection{The case of $\alpha$- and $\beta$-monotone operators}

Convergence of the adaptive DR algorithm for $\alpha$-monotone and $\beta$-monotone operators was provided in \cite[Section~4]{DP18a}. In this section we revisit some of these results by employing conical averagedness. In comparison to \cite{DP18a}, our new results (see~Theorem~\ref{t:2mono}) extend the admissible range for the parameters $\gamma,\delta,\lambda,\mu$, and $\kappa$, which guarantees the averagedness of $T$.

Similarly to Lemma~\ref{l:exist'}, we begin our discussion with the existence of parameters.

\begin{lemma}[existence of parameters]
\label{l:exist_para}
Let $\alpha,\beta\in \RR$ be such that $\alpha+\beta\geq 0$, and let $\gamma, \delta\in \RPP$. Set $\gamma_0$ as in \eqref{e:gamma0} and set $\Delta :=\gamma(1+\gamma\alpha)(\alpha+\beta)$. Then the following assertions are equivalent:
\begin{enumerate}
\item\label{l:exist_para-i} 
$(\gamma+\delta)^2 \leq 4\gamma\delta(1+\gamma\alpha)(1+\delta\beta)$.

\item\label{l:exist_para-ii}
$\frac{1}{\gamma} >\gamma_0$ and $\frac{1}{\gamma}(1+2\gamma\alpha-2\sqrt{\Delta}) \leq \frac{1}{\delta} \leq \frac{1}{\gamma}(1+2\gamma\alpha+2\sqrt{\Delta})$.
\end{enumerate}
Consequently, if $\alpha+\beta =0$, then \ref{l:exist_para-i} and \ref{l:exist_para-ii} are equivalent to $1+2\gamma\alpha >0$ and $\delta =\frac{\gamma}{1+2\gamma\alpha}$; if $\alpha+\beta >0$, then there always exist $\gamma, \delta\in \RPP$ such that
\begin{equation}
\frac{1}{\gamma} >\gamma_0 \quad\text{and}\quad
\max\left\{0,\frac{1}{\gamma}(1+2\gamma\alpha-2\sqrt{\Delta})\right\} <\frac{1}{\delta} <\frac{1}{\gamma}(1+2\gamma\alpha+2\sqrt{\Delta}),
\end{equation}
in which case, all inequalities in \ref{l:exist_para-i} and \ref{l:exist_para-ii} are strict. Moreover, if \ref{l:exist_para-i} or \ref{l:exist_para-ii} holds, then $1+\gamma\alpha >0$ and $1+\delta\beta >0$.
\end{lemma}
\begin{proof}
We note that \ref{l:exist_para-i} is equivalent to 
\begin{equation}
\left(\frac{1}{\gamma}+\frac{1}{\delta}\right)^2 \leq 4\left(\frac{1}{\gamma}+\alpha\right)\left(\frac{1}{\delta}+\beta\right),
\end{equation} 
while the last two inequalities in \ref{l:exist_para-ii} can be written in the form
\begin{equation}
\frac{1}{\gamma}+2\alpha -2\sqrt{\left(\frac{1}{\gamma}+\alpha\right)(\alpha+\beta)} \leq \frac{1}{\delta} \leq \frac{1}{\gamma}+2\alpha +2\sqrt{\left(\frac{1}{\gamma}+\alpha\right)(\alpha+\beta)}.
\end{equation} 
The proof is complete by invoking Lemma~\ref{l:exist'} with $\frac{1}{\gamma}$ and $\frac{1}{\delta}$.
\end{proof}

\begin{theorem}[aDR for $\alpha$-monotone and $\beta$-monotone operators]
\label{t:2mono}
Suppose that $A$ and $B$ are maximally $\alpha$-monotone and $\beta$-monotone, respectively, that $\lambda,\mu$ are set by \eqref{e:CQ}, and that either 
\begin{enumerate}
\item\label{t:2mono_0}
$\alpha+\beta =0$, $1+2\gamma\alpha >0$, $\delta =\frac{\gamma}{1+2\gamma\alpha}$, and $\kappa^* :=1$; or
\item\label{t:2mono_strong}
$\alpha+\beta >0$ and $\kappa^*:=\frac{4\gamma\delta(1+\gamma\alpha)(1+\delta\beta) -(\gamma+\delta)^2}{2\gamma\delta(\gamma+\delta)(\alpha+\beta)} >0$.
\end{enumerate}
Then the adaptive DR operators $T_{A,B}$ and $T_{B,A}$ are conically $\frac{\kappa}{\kappa^*}$-averaged and have full domain. Moreover, if $\zer(A+B)\neq \varnothing$ and $\kappa <\kappa^*$, then, for any $(T,C)\in \{(T_{A,B},\gamma A), (T_{B,A},\delta B)\}$, every sequence $(x_n)_\nnn$ generated by $T$ converges weakly to a point $\overline{x}\in \Fix T$ with $J_{C}\overline{x}\in \zer(A+B)$ and the rate of asymptotic regularity of $T$ is $o(1/\sqrt{n})$.
\end{theorem}
\begin{proof}
The existence of $\gamma, \delta\in \RPP$ in \ref{t:2mono_0} is clear while in \ref{t:2mono_strong} it follows from the existence of $\gamma, \delta\in \RPP$ such that
\begin{equation}
(\gamma+\delta)^2 <4\gamma\delta(1+\gamma\alpha)(1+\delta\beta)
\end{equation}
due to Lemma~\ref{l:exist_para}. Furthermore, in both cases \ref{t:2mono_0} and \ref{t:2mono_strong}, Lemma~\ref{l:exist_para} implies that $1+\gamma\alpha >0$ and $1+\delta\beta >0$. On the one hand, by Proposition~\ref{p:resol-mono}, $J_{\gamma A},J_{\delta B}$ and, hence, $T_{A,B}$ and $T_{B,A}$ are single-valued and have full domain. On the other hand, Proposition~\ref{p:rresol-mono} implies that $\frac{1}{1-\lambda}R_1$ and $\frac{1}{1-\mu}R_2$ are conically $\theta_1$-averaged and conically $\theta_2$-averaged, respectively, where
\begin{equation}
\theta_1:=\frac{\lambda}{2(\lambda-1)(1+\gamma\alpha)}
=\frac{\gamma+\delta}{2\delta(1+\gamma\alpha)}
\text{~~and~~}
\theta_2:=
\frac{\mu}{2(\mu-1)(1+\delta\beta)}
=\frac{\gamma+\delta}{2\gamma(1+\delta\beta)}.
\end{equation}
Next, we observe that if \ref{t:2mono_0} holds, then $\theta_1 = \theta_2 =1$ and $1 =\frac{1}{\kappa^*}$; if \ref{t:2mono_strong} holds, then $\theta_1\theta_2 <1$ and 
\begin{equation}
\frac{\theta_1+\theta_2-2\theta_1\theta_2}{1-\theta_1\theta_2} =\frac{\frac{1}{\theta_1}+\frac{1}{\theta_2}-2}{\frac{1}{\theta_1\theta_2}-1} =\frac{2\gamma\delta(\gamma+\delta)(\alpha+\beta)}
{4\gamma\delta(1+\gamma\alpha)(1+\delta\beta)-(\gamma+\delta)^2}
=\frac{1}{\kappa^*}.
\end{equation}
By applying Corollary~\ref{c:2averaged-equiv} to $\frac{1}{1-\lambda}R_1$ and $\frac{1}{1-\mu}R_2$, we conclude that $R_2R_1 =\frac{1}{(1-\lambda)(1-\mu)}R_2R_1$ and $R_1R_2 =\frac{1}{(1-\lambda)(1-\mu)}R_1R_2$ are conically $\frac{1}{\kappa^*}$-averaged. By invoking Proposition~\ref{p:ext_avg}\ref{p:ext_avg_ori}\&\ref{p:ext_avg_rlx}, $T_{A,B} =(1-\kappa)\Id+\kappa R_2R_1$ and $T_{B,A} =(1-\kappa)\Id+\kappa R_1R_2$ are conically $\frac{\kappa}{\kappa^*}$-averaged. 
Finally, by recalling Proposition~\ref{p:cvg}, the proof is complete. 
\end{proof}

Theorem~\ref{t:2mono} sheds new light on the convergence analysis of the adaptive DR algorithm for two generalized monotone operators \cite[Section~4]{DP18a}.
More specifically, Theorem~\ref{t:2mono}\ref{t:2mono_0} recovers the convergence of the classical DR algorithm for two maximally monotone operators, see, e.g., \cite{LM79}. Although Theorem~\ref{t:2mono}\ref{t:2mono_0} is the case $\alpha+\beta=0$ in \cite[Theorem~4.5]{DP18a}, our current alternative proof employs conical averagedness and applies concurrently to both adaptive DR operators $T_{A,B}$ and $T_{B,A}$. In addition, Theorem~\ref{t:2mono}\ref{t:2mono_strong} presents an improvement for \cite[Theorem~4.5(i)]{DP18a} on the parameter range, as shown in the following remark.

\begin{remark}[aDR for the case $\alpha+\beta>0$]
Theorem~\ref{t:2mono}\ref{t:2mono_strong} readily implies and extends \cite[Theorem~4.5(i)]{DP18a} in terms of parameter ranges. Indeed, in order to obtain convergence of the adaptive DR algorithm in the case where $\alpha+\beta >0$, in Theorem~\ref{t:2mono}\ref{t:2mono_strong} it is only required that 
\begin{equation}
0 <\kappa <\kappa^*,
\end{equation} 
while the assumptions of \cite[Theorem~4.5(i)]{DP18a} are equivalent to the more restrictive conditions
\begin{equation}\label{e:pre assumptions}
1+2\gamma\alpha>0 \quad\text{and}\quad 0 <\kappa <1 \leq\kappa^*.
\end{equation}
In order to verify \eqref{e:pre assumptions}, we first claim that 
\begin{equation}\label{e:kappa*}
\kappa^* \geq 1 
\iff \frac{1-2\gamma\beta}{\gamma} \leq \frac{1}{\delta} \leq \frac{1+2\gamma\alpha}{\gamma}
\ \iff \mu\in \left[2-2\gamma\beta, 2+2\gamma\alpha\right].
\end{equation}
Since $\alpha+\beta >0$, it follows that
\begin{subequations}
\begin{align}
\kappa^* \geq 1 
&\iff 2\gamma\delta(\gamma+\delta)(\alpha+\beta) \leq 4\gamma\delta(1+\gamma\alpha)(1+\delta\beta) -(\gamma+\delta)^2\\
&\iff (1+2\gamma\alpha-2\gamma\beta-4\gamma^2\alpha\beta)\delta^2 -2\gamma(1+\gamma\alpha-\gamma\beta)\delta +\gamma^2 \leq 0\\
&\iff (1+2\gamma\alpha-2\gamma\beta-4\gamma^2\alpha\beta) -2(1+\gamma\alpha-\gamma\beta)\frac{\gamma}{\delta} +\frac{\gamma^2}{\delta^2} \leq 0\\
&\iff 1-2\gamma\beta \leq \frac{\gamma}{\delta} \leq 1+2\gamma\alpha.
\end{align}
\end{subequations}
Since $\mu =1+\frac{\gamma}{\delta}$, we arrive at \eqref{e:kappa*}.

In the case where $\alpha+\beta >0$, we recall that \cite[Theorem~4.5(i)]{DP18a} requires $(\gamma,\delta,\lambda,\mu)\in \RPP^2\times\left]1,+\infty\right[^2$ such that
\begin{subequations}\label{e:DP18a1}
\begin{align}
&1+2\gamma\alpha>0,\\
&\mu\in\left[2-2\gamma\beta,2+2\gamma\alpha\right], \label{e:DP18a1-b}\\
&(\lambda-1)(\mu-1)=1,\ \delta=(\lambda-1)\gamma. \label{e:DP18a1-c}
\end{align}
\end{subequations}
and $\kappa\in \left]0,1\right[$. We note that \eqref{e:DP18a1-c} is equivalent to \eqref{e:CQ}, and that \eqref{e:DP18a1-b} is equivalent to $\kappa^* \geq 1$ due to \eqref{e:kappa*}. Therefore, the assumptions of \cite[Theorem~4.5(i)]{DP18a} for the case $\alpha+\beta>0$ can be rewritten as \eqref{e:pre assumptions}. 
\end{remark}

The following corollary recovers the main convergence result in \cite[Theorem~4.5(ii)]{DP18a}.

\begin{corollary}[DR for $\alpha$-monotone and $\beta$-monotone operators]
\label{c:dr_mono}
Suppose that $A$ and $B$ are maximally $\alpha$-monotone and $\beta$-monotone, respectively, that $\gamma =\delta\in \RPP$ and $\lambda =\mu =2$, and that either
\begin{enumerate}
\item 
$\alpha = \beta =0$ and $\kappa^* :=1$; or
\item 
$\alpha+\beta >0$ and $\kappa^* :=1+\gamma\frac{\alpha\beta}{\alpha+\beta} >0$.
\end{enumerate}
Then the DR operators $T_{A,B}$ and $T_{B,A}$ are conically $\frac{\kappa}{\kappa^*}$-averaged and have full domain. Moreover, if $\zer(A+B)\neq \varnothing$ and $\kappa <\kappa^*$, then, for any $(T,C)\in \{(T_{A,B},\gamma A), (T_{B,A},\delta B)\}$, every sequence $(x_n)_\nnn$ generated by $T$ converges weakly to a point $\overline{x}\in \Fix T$ with $J_{C}\overline{x}\in \zer(A+B)$ and the rate of asymptotic regularity of $T$ is $o(1/\sqrt{n})$.
\end{corollary}
\begin{proof}
We employ Theorem~\ref{t:2mono} with $\gamma =\delta$ and $\lambda =\mu =2$.
\end{proof}

\begin{remark}[DR for the case $\alpha+\beta>0$]
In the classical setup, the DR operator is defined as a strict convex combination of $\Id$ and $R_2R_1$ (or $R_1R_2$), i.e., $\kappa\in\left]0,1\right[$. In Corollary~\ref{c:dr_mono}, suppose that both operators $A$ and $B$ are strongly monotone ($\alpha>0$ and $\beta>0$), then the upper bound for $\kappa$ is $\kappa^*=1+\gamma\frac{\alpha\beta}{\alpha+\beta}>1$ which implies that $\kappa$ can be chosen to be larger than $1$ and the DR still converges. On the other hand, if either $A$ or $B$ is weakly monotone (i.e., $\alpha<0$ or $\beta<0$), then $\kappa^*<1$, which means that one needs to further restrict $\kappa$ than the standard range $\left]0,1\right[$ in order to obtain the convergence.
\end{remark}

For the remainder of this section we consider the adaptive DR algorithm for the problem of minimizing the sum of two functions. Let $f\colon X\to \left]-\infty,+\infty\right]$. Then $f$ is \emph{proper} if $\dom f :=\menge{x\in X}{f(x) <+\infty}\neq \varnothing$, and \emph{lower semicontinuous} if $\forall x\in \dom f$, $f(x)\leq \liminf_{z\to x} f(z)$. 
Set $\alpha\in \RR$. The function $f$ is \emph{$\alpha$-convex} (see \cite[Definition~4.1]{Vial83}) if $\forall x,y\in\dom f$, $\forall\kappa\in \left]0,1\right[$,
\begin{equation}
\label{e:alpha-cvx}
f((1-\kappa) x+\kappa y) +\frac{\alpha}{2}\kappa(1-\kappa)\|x-y\|^2\leq (1-\kappa)f(x)+\kappa f(y).
\end{equation}
If~\eqref{e:alpha-cvx} holds with $\alpha=0$, then we say that $f$ is \emph{convex}, $\alpha >0$, then we say that $f$ is strongly convex, $\alpha<0$, then we say that $f$ is \emph{weakly convex}. The \emph{proximity operator} of a proper function $f$ with parameter $\gamma\in \RPP$ is the mapping $\prox_{\gamma f}\colon X\rightrightarrows X$ defined by
\begin{equation}
\label{e:prox}
\forall x\in X, \quad \prox_{\gamma f}(x) :=\argmin_{z\in X}\left(f(z)+\frac{1}{2\gamma}\|z-x\|^2\right).
\end{equation}
Finally, we consider the \emph{$(\alpha,\beta)$-convex minimization} problem of the form 
\begin{equation}
\label{e:minsum}
\min_{x\in X} \big(f(x)+g(x)\big),
\end{equation}
where $f$ and $g$ are $\alpha$-convex and $\beta$-convex functions, respectively. This problem arises in several important applications, see, for example, \cite{GHY17}. The adaptive DR algorithm for problem \eqref{e:minsum} incorporates the operators
\begin{equation}\label{e:T_prox}
T_{f,g} :=(1-\kappa)\Id+\kappa R_gR_f 
\text{~~and~~} 
T_{g,f} :=(1-\kappa)\Id+\kappa R_fR_g,
\end{equation}
where
\begin{equation}
R_f :=(1-\lambda)\Id+\lambda\prox_{\gamma f}
\text{~~and~~} 
R_g :=(1-\mu)\Id+\mu\prox_{\delta g}.
\end{equation}

We continue our analysis in order to show that the relations between the proximal operators, the resolvents of the subdifferentials and the problem of finding zeros of the sum of the operators are useful in producing a solution of the minimization problem. To this end we recall that the Fr\'echet subdifferential of $f$ at $x$ is defined by 
\begin{equation}
\label{e:Frechet}
\widehat{\partial}f(x) :=\Menge{u\in X}{\liminf_{z\to x}\frac{f(z)-f(x)-\scal{u}{z-x}}{\|z-x\|}\geq 0}.
\end{equation}
For additional details on various subdifferentials and related properties, we refer the reader to the monograph \cite{Mor06}.

\begin{lemma}[proximity operators of $\alpha$-convex functions]
\label{l:prox}
Let $f\colon X\to \left]-\infty,+\infty\right]$ be a proper, lower semicontinuous and $\alpha$-convex function. Let $\gamma\in \RPP$ be such that $1+\gamma\alpha >0$. Then
\begin{enumerate}
\item\label{l:prox_diff} 
$\widehat{\partial}f$ is maximally $\alpha$-monotone.
\item\label{l:prox_equal} 
$\prox_{\gamma f} =J_{\gamma\widehat{\partial}f}$ is single-valued and has full domain.
\end{enumerate}
\end{lemma}
\begin{proof}
See~\cite[Lemma~5.2]{DP18a}.
\end{proof}

\begin{theorem}[aDR for $(\alpha,\beta)$-convex minimization]
\label{t:f-cvg}
Let $f\colon X\to\left]-\infty,+\infty\right]$ 
and $g\colon X\to \left]-\infty,+\infty\right]$ 
be proper and lower semicontinuous. Suppose that $f$ and $g$ are $\alpha$-convex and $\beta$-convex, respectively, that $\lambda,\mu$ are set by \eqref{e:CQ}, and that either
\begin{enumerate}
\item
$\alpha+\beta =0$, $1+2\gamma\alpha >0$, $\delta =\frac{\gamma}{1+2\gamma\alpha}$, $\kappa^* :=1$; or
\item
$\alpha+\beta >0$, $\kappa^* :=\frac{4\gamma\delta(1+\gamma\alpha)(1+\delta\beta) -(\gamma+\delta)^2}{2\gamma\delta(\gamma+\delta)(\alpha+\beta)} >0$.
\end{enumerate}
Then the adaptive DR operators $T_{f,g}$ and $T_{g,f}$ are conically $\frac{\kappa}{\kappa^*}$-averaged and have full domain.
Moreover, if $\zer(\widehat\partial f+\widehat\partial g)\neq \varnothing$ and $\kappa <\kappa^*$, then, for any $(T,h)\in \{(T_{f,g},\gamma f), (T_{g,f},\delta g)\}$, every sequence $(x_n)_\nnn$ generated by $T$ converges weakly to a point $\overline{x}\in \Fix T$ with $\prox_{h}(\overline{x})\in \zer(\widehat{\partial}f+\widehat{\partial}g)\subseteq \argmin(f+g)$ and the rate of asymptotic regularity of $T$ is $o(1/\sqrt{n})$.
\end{theorem}
\begin{proof}
Similarly to the proof of Theorem~\ref{t:2mono}, our assumptions imply that $1+\gamma\alpha >0$ and $1+\delta\beta >0$. Now, Lemma~\ref{l:prox} implies that $\widehat{\partial}f$ and $\widehat{\partial}g$ are maximally $\alpha$-monotone and $\beta$-monotone, respectively, where $\prox_{\gamma f} =J_{\gamma\widehat{\partial}f}$ and $\prox_{\gamma g} =J_{\gamma\widehat{\partial}g}$. We also have from \cite[Lemma~5.3]{DP18a} that $\zer(\widehat{\partial}f+\widehat{\partial}g)\subseteq \argmin(f+g)$. The conclusion then follows by applying Theorem~\ref{t:2mono} to $A =\widehat{\partial}f$ and $B =\widehat{\partial}g$.  
\end{proof}

As a consequence, we retrieve \cite[Theorem~5.4(ii)]{DP18a} which, in turn, unifies and extends \cite[Theorems~4.4 and~4.6]{GHY17} to the Hilbert space setting. 
\begin{corollary}[DR for $(\alpha,\beta)$-convex minimization]
\label{c:f-cvg}
Let $f\colon X\to\left]-\infty,+\infty\right]$ 
and $g\colon X\to \left]-\infty,+\infty\right]$ 
be proper and lower semicontinuous. Suppose that $f$ and $g$ are $\alpha$-convex and $\beta$-convex, respectively, that $\gamma =\delta\in \RPP,\ \lambda =\mu =2$, and that either 
\begin{enumerate}
\item
$\alpha =\beta =0$, $\kappa^* :=1$; or
\item
$\alpha+\beta >0$, $\kappa^* :=1 +\gamma\frac{\alpha\beta}{\alpha+\beta} >0$.
\end{enumerate}
Then the DR operators $T_{f,g}$ and $T_{g,f}$ are conically $\frac{\kappa}{\kappa^*}$-averaged and have full domain.
Moreover, if $\zer(\widehat\partial f+\widehat\partial g)\neq \varnothing$ and $\kappa <\kappa^*$, then, for any $(T,h)\in \{(T_{f,g},\gamma f), (T_{g,f},\delta g)\}$, every sequence $(x_n)_\nnn$ generated by $T$ converges weakly to a point $\overline{x}\in \Fix T$ with $\prox_{h}(\overline{x})\in \zer(\widehat{\partial}f+\widehat{\partial}g)\subseteq \argmin(f+g)$ and the rate of asymptotic regularity of $T$ is $o(1/\sqrt{n})$.
\end{corollary}
\begin{proof}
Apply Theorem~\ref{t:f-cvg} with $\gamma =\delta$ and $\lambda =\mu =2$.
\end{proof}

\section*{Acknowledgement}
We would like to thank two anonymous referees for their careful review and helpful comments. SB was partially supported by a UMass Lowell faculty startup grant. MND was partially supported by Discovery Projects 160101537 and 190100555 from the Australian Research Council. HMP was partially supported by Autodesk, Inc. via a gift made to the Department of Mathematical Sciences, UMass Lowell.

\end{document}